\documentclass[11pt,a4paper,oneside,reqno]{amsart}

\usepackage{amsmath, amsthm, amsopn, amssymb, amsfonts}
\usepackage{bbm}
\usepackage{enumerate, enumitem}
\usepackage{stmaryrd} 
\usepackage{mathtools} 
\usepackage{calrsfs}
\usepackage{color}
\usepackage[normalem]{ulem} 
\usepackage{nicefrac}
\usepackage{tikz}
\definecolor{myblue}{rgb}{0.09,0.32,0.44} 
\usepackage[colorlinks=true, linkcolor=myblue, citecolor=myblue,urlcolor=blue,pdfborder={0 0 0},pdfborderstyle={}]{hyperref}

\usepackage{geometry}
\newtheorem{thm}{Theorem}
\newtheorem{lemma}[thm]{Lemma}
\newtheorem{prop}[thm]{Proposition}
\newtheorem{claim}[thm]{Claim}

\theoremstyle{definition}

\newtheorem*{remark}{Remark}

\newcommand{\Aut}{\mathrm{Aut}}
\newcommand{\cA}{\mathcal{A}}

\newcommand{\cG}{\mathcal{G}}
\newcommand{\cH}{\mathcal{H}}
\newcommand{\cJ}{\mathcal{J}}
\newcommand{\cL}{\mathcal{L}}
\newcommand{\cP}{\mathcal{P}}
\newcommand{\cQ}{\mathcal{Q}}

\newcommand{\cX}{\mathcal{X}}
\newcommand{\cY}{\mathcal{Y}}

\newcommand{\Reals}{\mathbb{R}}

\newcommand{\DKL}{D_{KL}}

\newcommand{\Err}{\mathcal{E}}
\newcommand{\Disc}{D}

\newcommand{\Ber}{\mathrm{Ber}}
\newcommand{\Bin}{\mathrm{Bin}}

\newcommand{\Ex}{\mathbb{E}}
\newcommand{\Var}{\mathrm{Var}}

\newcommand{\eps}{\varepsilon}

\newcommand{\gain}{G}
\newcommand{\loss}{L}

\newcommand{\Yq}{{Y^{(q)}}}

\newcommand{\one}{\mathbf{1}}

\newcommand{\br}[1]{\llbracket{#1}\rrbracket}

\newcommand{\sumone}{I}
\newcommand{\sumtwo}{II}

\def\dtv{d_\textrm{TV}}

\renewcommand{\le}{\leqslant}
\renewcommand{\ge}{\geqslant}

\renewcommand{\Pr}{\mathbb{P}}

\title[Lower tails via relative entropy]{Lower tails via relative entropy}
\author{Gady Kozma}
\address{Department of Mathematics, The Weizmann Institute of Science, Rehovot 7610001, Israel}
\email{gady.kozma@weizmann.ac.il}

\author{Wojciech Samotij}
\address{School of Mathematical Sciences, Tel Aviv University, Tel Aviv 6997801, Israel}
\email{samotij@tauex.tau.ac.il}

\thanks{This research was supported in part by the Jesselson Foundation and by Paul and Tina Gardner (GK) and by the Israel Science Foundation grant 1145/18 (WS)}

\begin{document}

\begin{abstract}
  We show that the naive mean-field approximation correctly predicts the leading term of the logarithmic lower tail probabilities for the number of copies of a given subgraph in $G(n,p)$ and of arithmetic progressions of a given length in random subsets of the integers in the entire range of densities where the mean-field approximation is viable.

    Our main technical result provides sufficient conditions on the maximum degrees of a uniform hypergraph $\cH$ that guarantee that the logarithmic lower tail probabilities for the number of edges induced by a binomial random subset of the vertices of $\cH$ can be well-approximated by considering only product distributions. This may be interpreted as a weak, probabilistic version of the hypergraph container lemma that is applicable to all sparser-than-average (and not only independent) sets.
\end{abstract}

\maketitle

\section{Introduction}

This paper is concerned with the phenomenon that, in many cases, conditioning on an atypical event leads to a mixture of product measures.  An emblematic is the family of $n$-vertex graphs with no triangles.  It is clear that if one divides $\br{n} \coloneqq \{1,\dotsc,n\}$ into two parts and takes only edges with one endpoint in each part, the resulting graph has no triangles.  The classical result of Erd\H{o}s, Kleitman, and Rothschild~\cite{ErdKleRot76} states that the vast majority of triangle-free graphs have such simple structure.  In other words, if we \emph{condition} the random graph $G(n, \frac{1}{2})$ to have no triangles, the resulting measure can be approximated by the following process:  First, choose a random partition of the vertices into two parts (according to a measure that strongly favours partitions into approximately equal parts).  Then, choose the edges randomly and independently, with edges between the parts having probability $\frac{1}{2}$ and edges inside the parts having probability~$0$.  Since, conditioned on the partition, the measure becomes a product measure, the overall process is called a mixture of product measures.

The aim of this work is to establish sufficient conditions for such a phenomenon to occur in the context of large deviations for subgraph counts in the binomial random graph $G(n,p)$.  Here, the seminal work of Chatterjee and Varadhan~\cite{ChaVar11} has clarified that there are in fact two independent steps involved.  The first is to show that the distribution of the random graph conditioned on a tail event can be described by a (small) mixture of product measures.  The second is to describe the relevant measures, which, as it turns out, are those among all product measures (essentially) supported on the relevant tail event that have the least entropic cost.  The main result of~\cite{ChaVar11} completes the first of these two steps and can be summarised\footnote{The setting of~\cite{ChaVar11} is more general, but let us not state the full generality of that paper here.} as follows: Denote by $X_n$ the number of copies of a given graph in the binomial random graph $G(n,p)$. If the edge probability $p$ is fixed and $n$ tends to infinity, then
\begin{equation}
  \label{eq:ChaVar-LDP}
  - \log \Pr\big(X_n\ge(1+\delta)\Ex[X_n]\big) = (1+o(1)) \cdot \Phi_{n, p}(\delta),
\end{equation}
where  $\Phi_{n, p}(\delta)$ is the least entropic cost of a product measure supported on the upper tail event (we will give a formal definition below); the analogous result holds for the lower tail. As for the second step, the problem of calculating $\Phi_{n,p}(\delta)$ turned out to be very difficult. Even in the seemingly simple case of triangle counts, only partial results are known~\cite{LubZha15, Zha17}. \emph{In this paper, we address only the first step, namely, obtaining an identity akin to~\eqref{eq:ChaVar-LDP}.}

A substantial drawback of the approach taken by~\cite{ChaVar11}, which is based on Szemer\'edi's regularity lemma, is that it does not extend to sparse random graphs.  (One may instead use the so-called weak regularity lemma of Frieze and Kannan~\cite{FriKan99}, but this allows one to extend~\eqref{eq:ChaVar-LDP} only to the regime $p \ge (\log n)^{-c}$, for some small positive constant $c$, see~\cite{LubZha17}.) This was first rectified by the breakthrough work of Chatterjee and Dembo~\cite{ChaDem16}, who developed a general technique for computing large deviation probabilities of nonlinear functions of independent Bernoulli random variables, such as subgraph counts in $G(n,p)$. In the context of subgraph counts in $G(n,p)$, the general result of~\cite{ChaDem16} implies that~\eqref{eq:ChaVar-LDP} continues to hold as long as $p \ge n^{-\alpha}$ for some $\alpha > 0$ that depends only on the graph whose copies are counted.

The paper of Chatterjee and Dembo inspired a series of further developments. Their general technique was further simplified and strengthened by Eldan~\cite{Eld18}.  In the context of upper tails for subgraph counts in $G(n,p)$, the range of validity of the approximation~\eqref{eq:ChaVar-LDP} was further extended by the works of Augeri~\cite{Aug20} (for cycles), of Cook and Dembo~\cite{CooDem20} (for arbitrary graphs), and of Cook, Dembo, and Pham~\cite{CooDemPha21} (for arbitrary graphs and, more generally, arbitrary uniform hypergraphs). The expression $\Phi_{n, p}(\delta)$ in the right-hand side of~\eqref{eq:ChaVar-LDP} was computed in the range $n^{-1/\Delta} \ll p \ll 1$, where $\Delta$ is the maximum degree of the graph for cliques~\cite{LubZha17} and, subsequently, for arbitrary subgraphs~\cite{BhaGanLubZha17}. A very different, combinatorial technique for computing upper tail probabilities of polynomials of independent Bernoulli random variables was recently developed by Harel, Mousset, and Samotij~\cite{HarMouSam19}.  This technique was used to resolve the upper tail problem completely for cliques~\cite{HarMouSam19} and, subsequently, for all regular graphs~\cite{BasBas19}. More precisely, these works showed that the approximation~\eqref{eq:ChaVar-LDP} is valid in the entire range of densities $p$ where it was expected to hold.

Let us stress that all of the works on large deviations of subgraph counts in sparse random graphs mentioned above were primarily concerned with the upper tail. (In fact, the techniques developed in both~\cite{Aug20} and~\cite{HarMouSam19} are inapplicable to the lower tail problem.) Historically, the upper tail problem is considered to be more difficult of the two. Whereas Janson, {\L}uczak, and Ruci\'nski~\cite{JanLucRuc90} determined the logarithm of the lower tail probability up to a multiplicative constant, for every graph and all densities $p$, already in the late 1980s, the order of magnitude of the logarithm of the upper tail probability in the special case of triangle counts was determined only around ten years ago~\cite{Cha12, DeMKah12}.

In this paper, we offer a new, entropy-based approach to the large deviation problem that is particularly effective in estimating lower tails. The idea of using entropy estimates for studying nonlinear large deviations was first used in~\cite{KozMeyPelSam27} (that paper is a few years older than the current one, unlike what one might think by examining arXiv submission dates). Ultimately it stems from Avez's entropy approach to study random walks and amenability, see~\cite{Ave76}. A straightforward corollary of our main technical result is that the analogue of~\eqref{eq:ChaVar-LDP} holds for counts of arbitrary subgraphs in $G(n,p)$ in the entire range where such an approximation was expected to be valid.

\subsection{New results}
\label{sec:new-results}

We start with a special case of our result for triangles and $p=\frac12$. Of course, this case is mostly covered by~\cite{ChaVar11}, but we will get to values of $p$ not covered by the literature in Theorem~\ref{thm:graphs} below.
We first state the minimisation problem that formalises the phrase `least entropic cost' in this setting. Given a function $q \colon \binom{\br{n}}{2} \to [0,1]$, let $G(n, q)$ denote the random graph obtained by retaining each edge $e$ of $K_n$ independently with probability $q_e$. For each $t \ge 0$, define
\[
  \cQ_t \coloneqq \left\{ q \in [0,1]^{\binom{\br{n}}{2}} : \Ex\left[N_{K_3}\big(G(n,q)\big)\right] \le t \right\},
\]
where $N_{K_3}(G)$ is the number of triangles in $G$, and
\begin{equation}\label{eq:Phi_n}
  \Phi_n(t) \coloneqq \min\left\{\sum_{e \in \binom{\br{n}}{2}} \Big(q_e\log q_e + (1-q_e)\log(1-q_e)+\log 2\Big) : q \in \cQ_t\right\}.
\end{equation}
Note that $q \log q + (1-q) \log (1-q) + \log 2$ is the difference in entropies of Bernoulli random variables with success probabilities $\frac12$ and $q$.
\begin{thm}
  \label{thm:triangles}
  Let $X_n$ denote the number of triangles in $G(n, \frac{1}{2})$. For every $n$ and every $t \ge 0$,
  \begin{equation}
    \label{eq:thm-triangles}
    \log \Pr(X_n \le t) \le - \Phi_n(t + n^{23/8}) + 2n^{15/8}.
  \end{equation}
\end{thm}
\begin{remark}
  Note that the two error terms in the above estimate are better than $o(n^3)$ and $o(n^2)$, respectively. In the remainder of this paper, we follow the literature and prove results with error terms in the corresponding estimates of the lower tail probabilities being inexplicit, but here we made an exception.
\end{remark}

We now formulate a general result concerning the lower tail of subgraph counts in $G(n,p)$.  In order to phrase the minimisation problem in the case $p \neq \frac12$, it is convenient to first define
\[
  i_p(q) \coloneqq q\log\frac{q}{p} + (1-q) \log \frac{1-q}{1-p}.
\]
Further, given graphs $H$ and $G$, let $N_H(G)$ denote the number of copies of $H$ in $G$. For every graph $H$, integer $n$, real $p \in (0,1)$, and every $\eta \in [0,1]$, let
\[
  \Phi_{n,p}^H(\eta) \coloneqq \min\left\{ \sum_{e \in \binom{\br{n}}{2}} i_p(q_e) : \Ex\left[N_H\big(G(n,q)\big)\right] \le \eta \cdot \Ex\left[N_H\big(G(n,p)\big)\right] \right\},
\]
where the minimum is taken over all $q\in[0,1]^{\binom{\br{n}}{2}}$, of course. Recall that the $2$-density of a graph $H$ is the quantity $m_2(H)$ defined as follows: If $H$ has at least two edges, then
\[
  m_2(H) \coloneqq \max\left\{\frac{e_F-1}{v_F-2} : F \subseteq H, e_F \ge 2\right\};
\]
otherwise, $m_2(H) \coloneqq \frac{1}{2}$. The notation $F\subseteq H$ here means that $F$ is a subgraph of $H$. For example, $m_2(
\begin{tikzpicture}[baseline]
  \fill[black] (0,0.2) circle [radius=1pt];
  \fill[black] (-0.12,0) circle [radius=1pt];
  \fill[black] (0.12,0) circle [radius=1pt];
  \draw (0,0.2) -- (-0.12,0) -- (0.12,0) -- (0,0.2);
\end{tikzpicture})=2$ but $m_2(
\begin{tikzpicture}[baseline]
  \fill[black] (0,0) circle [radius=1pt];
  \fill[black] (0,0.2) circle [radius=1pt];
  \fill[black] (0.2,0) circle [radius=1pt];
  \fill[black] (0.2,0.2) circle [radius=1pt];
  \fill[black] (0.4,0.2) circle [radius=1pt];
  \draw (0,0) -- (0,0.2) -- (0.2,0.2) -- (0.4,0.2);
  \draw (0,0) -- (0.2,0) -- (0.2,0.2) -- (0,0);
  \draw (0,0.2) -- (0.2,0);
\end{tikzpicture})=\frac52$ because the maximum is attained at the subgraph 
\begin{tikzpicture}[baseline]
  \fill[black] (0,0) circle [radius=1pt];
  \fill[black] (0,0.2) circle [radius=1pt];
  \fill[black] (0.2,0) circle [radius=1pt];
  \fill[black] (0.2,0.2) circle [radius=1pt];
  \draw (0,0) -- (0,0.2) -- (0.2,0.2);
  \draw (0,0) -- (0.2,0) -- (0.2,0.2) -- (0,0);
  \draw (0,0.2) -- (0.2,0);
\end{tikzpicture}.

\begin{thm}
  \label{thm:graphs}
  For every nonempty graph $H$, all $p_0 < 1$, and every $\eps > 0$, there exists a constant $L$ such that the following holds:  Suppose that $Ln^{-1/m_2(H)} \le p \le p_0$ and let $X \coloneqq N_H\big(G(n,p)\big)$.  Then, for every $\eta \in [0,1]$,
  \[
    (1-\eps) \cdot \Phi_{n,p}^H(\eta+\eps) \le - \log \Pr\big(X \le \eta \Ex[X]\big) \le (1 + \eps) \cdot \Phi_{n,p}^H\big((1-\eps)\eta\big).
  \]
\end{thm}

A key feature of Theorem~\ref{thm:graphs} is that the lower-bound assumption on $p$ is optimal up to constants.  To see this, note first that, by Harris's inequality, for every $F \subseteq H$,
\[
  \Pr(X = 0) = \Pr\big(H \nsubseteq G(n,p)\big) \ge \Pr\big(F \nsubseteq G(n,p)\big) \ge (1-p^{e_F})^{n^{v_F}} \ge \exp(-2n^{v_F}p^{e_F}).
\]
Moreover, $m_2(H)$ is defined so that $n^{v_F}p^{e_F} = o(n^2p)$ for some $F \subseteq H$ precisely when $p \ll n^{-1/m_2(H)}$. On the other hand, for all $H$, $n$, $p$, and $\eta < 1$, we have $\Phi_{n,p}^H(\eta) \ge c n^2p$ for some positive $c$ that depends only on $H$ and $\eta$ (see Lemma \ref{lemma:Phi-p-cH-lower} below).

The boundary case $\eta = 0$ in Theorem~\ref{thm:graphs}, the probability that a random graph is \emph{$H$-free}, has been extensively studied in the literature.  In particular, {\L}uczak~\cite{Luc00} computed the asymptotics of $\log \Pr\big(K_3 \nsubseteq G(n,p)\big)$ for all $p \gg n^{-1/m_2(K_3)}$ and derived an asymptotic formula for $\log \Pr\big(H \nsubseteq G(n,p)\big)$, for every nonbipartite graph $H$ and all $p \gg n^{-1/m_2(H)}$, from the so-called K{\L}R conjecture~\cite{KohLucRod97}, which was proved some fifteen years later by Balogh, Morris, and Samotij~\cite{BalMorSam15} and by Saxton and Thomason~\cite{SaxTho15}.  In fact, the hypergraph container theorems proved in~\cite{BalMorSam15, SaxTho15} can be used to compute the asymptotics of the logarithms of these probabilities directly, using simple, well-known results in extremal graph theory, see~\cite[\S1.3]{BalMorSam15}.

Our methods allow us to generalise Theorem~\ref{thm:graphs} to $s$-uniform hypergraphs in a straightforward way. Suppose that $H$ is a nonempty $s$-uniform hypergraph. The $s$-density of $H$ is the quantity $m_s(H)$ defined as follows: If $H$ has at least two edges, then
\begin{equation}
  \label{eq:msH}
  m_s(H) \coloneqq \max\left\{\frac{e_F-1}{v_F-s} : F \subseteq H, e_F \ge 2\right\};
\end{equation}
otherwise, $m_s(H) \coloneqq \frac{1}{s}$. For every integer $n$, real $p \in (0,1)$, and every $\eta \in [0,1]$, we define $\Phi_{n,p}^H(\eta)$ analogously to the graph case:
\[
  \Phi_{n,p}^H(\eta) \coloneqq \min\left\{ \sum_{e \in \binom{\br{n}}{s}} i_p(q_e) : \Ex\left[N_H\big(G^{(s)}(n,q)\big)\right] \le \eta \cdot \Ex\left[N_H\big(G^{(s)}(n,p)\big)\right] \right\},
\]
where $G^{(s)}(n, q)$ is the binomial random $s$-uniform hypergraph with vertex set~$\br{n}$.

\begin{thm}
  \label{thm:hypergraphs}
  For every nonempty $s$-uniform hypergraph $H$, all $p_0 < 1$, and every $\eps > 0$, there exists a constant $L$ such that the following holds:  Suppose that $Ln^{-1/m_s(H)} \le p \le p_0$ and let $X \coloneqq N_H\big(G^{(s)}(n,p)\big)$.  Then, for every $\eta \in [0,1]$,
  \[
    (1-\eps) \cdot \Phi_{n,p}^H(\eta+\eps) \le - \log \Pr\big(X \le \eta \Ex[X]\big) \le (1 + \eps) \cdot \Phi_{n,p}^H\big((1-\eps)\eta\big).
  \]
\end{thm}

As in Theorem~\ref{thm:graphs}, the lower-bound assumption on $p$ in Theorem~\ref{thm:hypergraphs} is optimal and the asymptotics of $\log \Pr(X = 0)$ can be derived from the hypergraph container theorems.

The final application of our new entropy method is a solution to the lower tail problem for the number of arithmetic progression of a give length in a binomial random subset of $\br{n}$, which will demonstrate that symmetry is not crucial for our methods. Given a function $q \colon \br{n} \to [0,1]$, we denote by $\br{n}_q$ the random subset of $\br{n}$ obtained by independently retaining each $i \in \br{n}$ with probability $q_i$. For a positive integer $k$ and a set $I \subseteq \br{n}$, let $A_k(I)$ denote the number of $k$-term arithmetic progressions in $I$.

\begin{thm}
  \label{thm:APs}
  For every positive integer $k$, all $p_0 < 1$, and every $\eps > 0$, there exists a constant $L$ such that the following holds:  Suppose that $Ln^{-1/(k-1)} \le p \le p_0$ and let $X \coloneqq A_k\big(\br{n}_p\big)$.  Then, for every $\eta \in [0,1]$,
  \[
    (1-\eps) \cdot \Phi_{n,p}^k  (\eta+\eps) \le - \log \Pr\big(X \le \eta \Ex[X]\big) \le (1 + \eps) \cdot \Phi_{n,p}^{k} \big((1-\eps)\eta\big).
  \]
\end{thm}

As before, the lower-bound assumption on $p$ in Theorem~\ref{thm:APs} is optimal and the asymptotics of $\log \Pr(X = 0)$ can be derived from the hypergraph container theorems, see~\cite[Theorem~1.1]{BalMorSam15}.

\subsection{The main technical result}
\label{sec:main-techn-result}

A natural way to generalise Theorems~\ref{thm:graphs}, \ref{thm:hypergraphs}, and \ref{thm:APs} is to represent the combinatorial objects we are counting as edges of an auxiliary hypergraph (no relation to the hypergraphs of Theorem \ref{thm:hypergraphs}). This way, each of the respective random variables counts the number of edges of such a hypergraph that are induced by random subset of its vertices. This idea is not new -- the transference principles of Conlon and Gowers~\cite{ConGow16} and Schacht~\cite{Sch16} and the hypergraph container theorems~\cite{BalMorSam15, SaxTho15} are prime examples of why taking such an abstract viewpoint may prove beneficial in our context. For example, in order to express the number of triangles in $G(n,p)$ this way, we consider the $3$-uniform hypergraph with vertex set $\binom{\br{n}}{2}$, the edge set of the complete graph on $\br{n}$, whose hyperedges are the $\binom{n}{3}$ triples of edges that form triangles in the complete graph on $\br{n}$.

We are thus led to ask the following general question: Given a hypergraph $\cH$ and a $p \in [0,1]$, what is the probability that a random subset of the vertices of $\cH$ formed by independently retaining each vertex with probability $p$ contains atypically few hyperedges? For extra generality, we allow the edges of the hypergraph to have positive weights.

Suppose that a hypergraph $\cH$ is equipped with a weight function $d \colon \cH \to (0, \infty)$. We shall denote by $e(\cH)$ the sum $\sum_{A \in \cH} d_A$ of all edge weights and, for every set $B \subseteq V(\cH)$, we shall write
\begin{equation}
  \label{eq:deg-cH-B}
  \deg_\cH B \coloneqq \sum_{B \subseteq A \in \cH} d_A.
\end{equation}
Moreover, for every $s \in \br{r}$, we define
\[
  \Delta_s(\cH) \coloneqq \max\left\{\deg_\cH B : B \subseteq V \text{ and } |B| = s\right\}.
\]
Note that when $d_A = 1$ for every $A \in \cH$, then we may simply view $\cH$ as a hypergraph; in this case, the above definitions give the usual notions of edge counts and degrees.

Let $\cH$ be a hypergraph and denote $V = V(\cH)$ for brevity. Let $Y = (Y_v)_{v \in V}$ be a sequence of i.i.d.\ Bernoulli random variables with success probability $p$, one for every vertex of the hypergraph $\cH$, and let $R$ be the corresponding random subset of $V$, i.e., $R \coloneqq \{v\in V:Y_v=1\}$. For a function $q \colon V \to [0,1]$, we let $\Yq = (Y_v')_{v \in V}$ be a sequence of independent Bernoulli random variables such that $Y_v'$ has success probability $q_v$ for each $v \in V$ and let $R^{(q)}$ be the corresponding random subset of $V$. For every nonnegative real $\eta$, define
\begin{equation}\label{eq:3.5}
    \Phi_p^{\cH}(\eta) \coloneqq \min\left\{\DKL\big(\Yq \,\|\, Y\big) : q \in [0,1]^V \text{ and } \Ex[e(\cH[R^{(q)}])]\le \eta \cdot \Ex[e(\cH[R])]\right\},
\end{equation}
where $\DKL$ is the Kullback--Leibler divergence, so that,
\[
  \DKL\big(\Yq \,\|\, Y\big) = \sum_{v \in V} i_p(q_v) = \sum_{v \in V} q_v \log\frac{q_v}{p} + (1-q_v) \log\frac{1-q_v}{1-p}.
\]
Here and below, $\cH[R]$ stands for the restriction of $\cH$ to $R$, namely the hypergraph whose vertices are $R$ and whose hyperedges are $\{A\in\cH:A\subseteq R\}$; thus $e(\cH[R])=\sum_{A\subseteq R}d_A$. Also, for $W \subseteq V(\cH)$, we write $\cH - W$ in place of $\cH[V(\cH) \setminus W]$.

\begin{thm}
  \label{thm:main-hypergraphs}
  For every integer $r$ and all $p_0<1$, $\eps>0$, and $K$, there exists a positive $\lambda$ and a $C$ such that the following holds. Let $V$ be a finite set and let $\cH$ be a nonempty $r$-uniform hypergraph with vertex set $V$ and weight function $d \colon \cH \to (0, \infty)$. Let $p \in (0, p_0]$ and let $R$ be the $p$-random subset of $V$. Suppose that, for every $s \in \br{r}$, the maximal degree $\Delta_s(\cH)$ satisfies
  \begin{equation}\label{eq:half}
    \Delta_s(\cH) \le K \cdot (\lambda p)^{s-1} \cdot \frac{e(\cH)}{v(\cH)}.
  \end{equation}
  Then, letting $X \coloneqq e(\cH[R])$, for every nonnegative real $\eta$,
  \[
    - \log \Pr\big(X \le \eta\Ex[X]\big) \ge (1-\eps) \Phi_p^{\cH}(\eta+\eps) - C.
  \]
\end{thm}

\begin{remark}
  Our argument gives the following explicit dependence of $\lambda$ and $C$ on the parameters:
  \[
    \lambda \ge 10^{-5} K^{-2} r^{-4} \eps^9 (1-p_0)
    \qquad
    \text{and}
    \qquad
    C \le 10^6K^2r^5\eps^{-9}(1-p_0)^{-1} \log\frac{1}{1-p_0}.
  \]
\end{remark}

The readers familiar with the hypergraph container method will likely notice striking similarities between the assumptions of Theorem~\ref{thm:main-hypergraphs} and the assumptions of the container lemmas proved in~\cite{BalMorSam15, BalSam20}. This is not a coincidence -- the boundary case $\eta = 0$ in Theorem~\ref{thm:main-hypergraphs} bounds the probability that the random set $R$ is independent in $\cH$ from above by $\Phi_p^{\cH}(\eps)$, a minimum over all distributions $q \in [0,1]^V$ such that $R^{(q)}$ induces at most $\eps e(\cH)$ edges in $\cH$, in expectation (cf.\ the combinatorial notion of containers for independent sets in~\cite{BalMorSam15, SaxTho15}).

While it might be tempting to replace $\Phi_p^{\cH}(\eta + \eps)$ in the assertion of Theorem~\ref{thm:main-hypergraphs} with $\Phi_p^{\cH}(\eta)$, or at least $\Phi_p^{\cH}((1+\eps)\eta)$, this is not always possible for $\eta$ very close to zero. To see this, observe first that $\Phi_p^{\cH}(0) = (\alpha(\cH) - v(\cH)) \cdot \log(1-p)$, where $\alpha(\cH)$ is the largest size of an independent set in~$\cH$. Suppose now that $\cH$ is the union of two hypergraphs with the same vertex set $V$: a dense hypergraph $\cH_1$ with $\alpha(\cH_1) \ge v(\cH)/2$ and a very sparse hypergraph $\cH_2$ with $\alpha(\cH_2) \le v(\cH)/4$. (For example, if $M$ and $v(\cH)$ are sufficiently large as a function of the uniformity $r$ only, then a random hypergraph with $M v(\cH)$ edges will typically have this property). Now, on the one hand,
\[
  \Phi_p^{\cH}(0) \ge \big(\alpha(\cH_2) - v(\cH)\big) \cdot \log(1-p) \ge \frac{3v(\cH)}{4} \cdot \log\frac{1}{1-p}
\]
but, on the other hand, by Harris's inequality, the $p$-random subset of some largest independent set of $\cH_1$ has probability at least $(1-p^r)^{e(\cH_2)}$ to be independent also in $\cH_2$ and thus, when $p$ is sufficiently small,
\[
  \Pr(X = 0) \ge (1-p)^{v(\cH) - \alpha(\cH_1)} \cdot e^{-2p^re(\cH_2)} \ge (1-p)^{2v(\cH)/3},
\]
showing that $-\log\Pr(X=0)$ is not close to $\Phi_p^{\cH}(0)$.

For easier comparison with the literature, let us reformulate Theorem~\ref{thm:main-hypergraphs} in the language of polynomials.  We retain the notations $V$, $Y$, and $Y^{(q)}$ as above.  We replace a weighted hypergraph $\cH$ with a homogeneous polynomial $f$ by turning each edge $A$ of $\cH$ with the monomial $d_A \cdot \prod_{v \in A}y_v$. The definition of $\Phi$ thus becomes
\[
  \Phi_p^f(\eta)=\min\left\{\DKL(Y^{(q)}\,\|\,Y):q\in[0,1]^V, \Ex[f(Y^{(q)})]\le \eta \cdot \Ex[f(Y)]\right\}.
\]
Moreover, the assumption~\eqref{eq:half} can be now expressed in terms of partial derivatives of $f$. Given a $B =\{v_1,\dotsc,v_k\}\subseteq V$ and a polynomial $f$ in $|V|$ variables, we denote
\[
\partial_B f \coloneqq\frac{\partial}{\partial v_1}\dotsb\frac{\partial}{\partial v_k}f.
\]
The following statement is a reformulation of Theorem~\ref{thm:main-hypergraphs}.
\begin{thm}
  \label{thm:main}
  For every integer $r$ and all  $p_0<1$, $\eps>0$, and $K$, there exists a positive $\lambda$ and a $C$ such that the following holds.
  Let $V$ be a finite set and let $f$ be an $r$-homogeneous $V$-variate multilinear polynomial with nonnegative coefficients. Let $p \in (0, p_0]$ and let $Y = (Y_v)_{v \in V}$ be a sequence of i.i.d.\ $\Ber(p)$ random variables. Suppose that, for every nonempty $B \subseteq V$ with $|B| \le r$,
  \[
    \partial_B f(\one) \le K \cdot (\lambda p)^{|B|-1} \cdot \frac{f(\one)}{|V|}.
  \]
  Then, letting $X \coloneqq f(Y)$, for every nonnegative real $\eta$,
  \[
    - \log \Pr\big(X \le \eta\Ex[X]\big) \ge (1-\eps) \Phi_p^f(\eta+\eps) - C.
  \]
\end{thm}

When $f$ is a linear function, the variable $X$ from the statement of the theorem is a sum of independent random variables. In this case, our argument can be simplified tremendously. The special case $f(y) = y_1 + \dotsb + y_n$, which corresponds to the binomial distribution, is treated in \S\ref{sec:interlude}, where a short, entropy-based proof of the optimal tail estimate
\[
  \Pr\big(\Bin(n,p) \le nq\big) \le \exp\big(-n \cdot i_p(q)\big)
\]
is given.

\subsection{Lower bounds on the lower tail}
\label{sec:lower-bound-lower}

We end the results section with a lower bound on the lower tail probabilities from the statements of Theorems~\ref{thm:main-hypergraphs} and~\ref{thm:main} that matches the upper bounds proved by these theorems. Since the proof of this lower bound is a relatively standard tilting argument, we relegate it to \S\ref{sec:lower}. Here is the exact formulation (in the language of Theorem~\ref{thm:main}).
\begin{thm}
  \label{thm:lower-bound}
  For every $p_0 < 1$ and $\eps > 0$, there exists a $C$ such that the following holds. Let $V$ be a finite set, let $Y = (Y_v)_{v \in V}$ be a sequence of i.i.d.\ $\Ber(p)$ random variables, let $f \colon \{0,1\}^V \to [0, \infty)$ be an arbitrary increasing function, and let $X \coloneqq f(Y)$. Then, for every nonnegative real $\eta$,
  \[
    - \log \Pr\big(X \le \eta \Ex[X] \big) \le (1+\eps) \Phi_p^f\big((1-\eps)\eta\big) + C.
  \]
\end{thm}

\subsection{Organisation of the paper}
\label{sec:organisation-paper}

The remainder of this paper is organised as follows. In Section~\ref{sec:proof-outline}, we outline of the proof of Theorem~\ref{thm:triangles} and discuss some of the additional ideas required in the proof of Theorem~\ref{thm:graphs} in the case $H = K_3$. In Section~\ref{sec:triangles}, which is merely three pages long, we present a complete proof of Theorem~\ref{thm:triangles}. In Section~\ref{sec:DKL}, we recall some basic properties of the Kullback--Leibler divergence and prove the key technical lemma (Lemma~\ref{lemma:main}) that relates independence and conditional KL-divergence. Subsection~\ref{sec:interlude} contains a short entropy-based proof of optimal tail bounds for binomial distributions, which might be of independent interest. Our main technical result, Theorem~\ref{thm:main-hypergraphs}, is proved in Section~\ref{sec:upper-bounds-lower-tail}. The matching lower bound for lower tail probabilities, Theorem~\ref{thm:lower-bound}, is proved in Section~\ref{sec:lower}. Finally, Section~\ref{sec:applications} contains short derivations of Theorems~\ref{thm:graphs}, \ref{thm:hypergraphs}, and~\ref{thm:APs}.

\section{Proof outline}
\label{sec:proof-outline}

\subsection{Triangle count in \texorpdfstring{$G(n,\frac{1}{2})$}{G(n,1/2)}}
\label{sec:triangles-Gn12}

Let us first explain how to prove Theorem~\ref{thm:triangles}, i.e., the upper bound on the lower tail of the number of triangles. Considering only $G(n, \tfrac{1}{2})$, which is the \emph{uniform} distribution on $n$-vertex graphs, allows us to phrase the argument in the familiar language of \emph{entropy} rather than using the Kullback--Leibler divergence. The argument sketched here is described in full in \S\ref{sec:triangles} and takes no more than three pages. 

Let $Y$ be the random graph $G(n,\frac12)$ \emph{conditioned} on having at most $t$ triangles. Then
\begin{equation}
  \label{eq:logPr-entropy-sketch}
  \log\Pr(X\le t)=H(Y)-\binom{n}{2}\log 2,
\end{equation}
where $H$ is the entropy of $Y$. Examine the distribution of the first edge (i.e., of $Y_{12}$) under the conditioning.\footnote{We assume that the vertex set of $G(n, \frac{1}{2})$ is $\br{n}$ and think of $Y \in \{0,1\}^{\binom{\br{n}}{2}}$ as the characteristic vector of the edge set of the conditioned random graph.} For every integer $r \ge 0$, let
\[
  h_r\coloneqq H(Y_{12} \mid \textrm{edges with at least one endpoint larger than $n-r$}),
\]
where $H(\cdot \mid \cdot)$ is the usual conditional entropy. Since conditional entropy is nonnegative and it decreases as one increases the conditioning, we have $0 \le h_{r+1}\le h_r$ for every $r$.  Since $h_0 = H(Y_{12}) \le \log 2$, there must be some $r \le \sqrt{n}$ such that $h_r-h_{r+1}\le C/\sqrt{n}$, where $C$ is an absolute constant.

Denote by $S_r$ the edges from the definition of $h_r$, so that $h_r=H(Y_{12} \mid S_r)$.  Since both $\{1, n-r\}$ and $\{2, n-r\}$ belong to $S_{r+1} \setminus S_r$, we may use the monotonicity of conditional entropy again to sandwich $H(Y_{12} \mid S_r, Y_{1,n-r}, Y_{2,n-r})$ between $h_r$ and $h_{r+1}$:
\[
  h_{r+1}=H(Y_{12} \mid S_{r+1})\le H(Y_{12} \mid Y_{1,{n-r}},Y_{2,n-r},S_r)\le H(Y_{12} \mid S_r)=h_r.
\]
Hence, we also get the inequality
\[
  H(Y_{12} \mid S_r)-H(Y_{12} \mid Y_{1,{n-r}},Y_{2,n-r},S_r)\le C/\sqrt{n}.
\]
By symmetry, we may replace $(1,2,n-r)$ in the above inequality with any three different elements $(i,j,k)$ of $\br{n-r}$ and get
\[
  H(Y_{ij} \mid S_r)-H(Y_{ij} \mid Y_{ik},Y_{jk},S_r)\le C/\sqrt{n}.
\]
We now apply \emph{Pinsker's inequality}, which states that, for any two variables $T$ and $U$, if $H(T)-H(T \mid U)$ is small, then $T$ and $U$ must be approximately independent.  We apply this to the variables $Y_{ij}$ conditioned on $S_r$ to conclude that, conditioned on $S_r$, the three edges of every triangle are (typically) approximately independent.

Recall now the definition of $\Phi_n(t)$. It is the minimum of $-H\big(G(n,q)\big)+\binom{n}{2}\log 2$ over all functions $q \colon \binom{\br{n}}{2} \to [0,1]$ such that
\[
  T(q) \coloneqq \Ex\big[\textrm{\#triangles in } G(n,q)\big] \le t.
\]
Consider the function $q_{ij}\coloneqq \Ex[Y_{ij} \mid S_r]$. The approximate independence of the $Y_{ij}$ gives that $T(q)$ is (typically) approximately the expected number of triangles in $Y$, which is at most $t$, by the definition of $Y$. Hence, $T(q)\le t+o(t)$, where the $o(t)$ error term comes from the fact that the $Y_{ij}$ are only approximately independent. We conclude that
\[
  H\big((Y_{ij})_{i,j\le n-r} \mid S_r\big) \le \sum_{i,j \le n-r} H(Y_{ij} \mid S_r) \le -\Phi_n(t+o(t))+\binom{n}2\log 2.
\]
Since $S_r$ has only at most $n^{3/2}$ edges, its entropy is negligible and we get
\[
  H(Y) = H(S_r) + H\big((Y_{ij})_{i,j\le n-r} \mid S_r\big) \le (1-o(1)) \cdot \Phi_n(t+o(t))+\binom{n}2\log 2,
\]
 as needed.

Examining the proof above, we see that the crucial step is that of proving \emph{conditional approximate independence}. Why was the conditioning necessary? Because $Y$ is not close to a product measure but rather a mixture of product measures. Heuristically, the conditioning chooses one product measure from the mixture.

\subsection{Triangle count in \texorpdfstring{$G(n,p)$}{G(n,p)} and beyond}
\label{sec:triangles-Gnp}

What is needed to prove Theorem~\ref{thm:triangles} with $G(n, \frac{1}{2})$ replaced by $G(n,p)$?
Since the latter is no longer a uniform distribution, in order to phrase a suitable analogue of~\eqref{eq:logPr-entropy-sketch}, we certainly have to replace entropy with entropy relative to a product of $p$-Bernoulli variables (relative entropy is also called the Kullback--Leibler divergence, though note that the sign of the Kullback--Leibler divergence is minus that of what a straightforward analogue of entropy would have) and we need an analogue of Pinsker's inequality for (conditional) relative entropy. These two ideas would have been enough to solve the lower tail (as well as the upper tail) problem for triangles in $G(n,p)$ for all $p \ge n^{-c}$, where $c$ is an absolute positive constant.

In order to extend the argument to all $p \gg n^{-1/2}$, one needs to prove a version of Pinsker's inequality that provides a stronger upper bound on the difference of probabilities that two measures assign to \emph{rare} events (rather than arbitrary events, as measured by the total variation distance). Furthermore, in order to use this strengthening of Pinsker's inequality, we also need to note that, when we condition our random graph on the lower tail event, the probability of every edge is at most $p$, even when we further condition on $S_r$. This follows from the Harris inequality (aka the FKG inequality). The use of Harris's inequality is the main (but not the only) reason why our methods are not as efficient for the upper tail problem.

The general setting of Theorem~\ref{thm:main-hypergraphs}, which lacks symmetry, requires a serious overhaul of the argument.  (Having said that, even in the setting of $K_4$ counts in $G(n,p)$, which still has a lot of symmetry, the argument sketched above does not work under the optimal assumption $p \gg n^{-2/5}$.)  We no longer increase the conditioning in small steps (recall the definition of $h_r$ above) but rather in large chunks, which are chosen randomly.  The crux of the matter is relating the decrease in entropy caused by conditioning on each such random chunk to approximate independence of the remaining variables.  Here, the key role is played by Lemma~\ref{lemma:main}, an improvement of Pinsker's inequality that is inspired by the statement of Janson's inequality~\cite{Jan90}.

\section{The lower tail of triangle count in \texorpdfstring{$G(n,\frac{1}{2})$}{G(n,1/2)}}
\label{sec:triangles}

As explained above, our proof of Theorem~\ref{thm:triangles} revolves around (information-theoretic) entropy.  For convenience
of the reader, we shall recall here the definitions of entropy and conditional entropy and list all of their
properties required for our argument; for proofs of these properties, we refer the reader to~\cite[Chapter~2]{CovTho06}.

\subsection{Preliminaries}
\label{sec:preliminaries-entropy}

The entropy of a random variable $X$ taking values in a finite set $\cX$ is the quantity $H(X)$
defined by
\[
  H(X) \coloneqq - \sum_{x \in \cX} \Pr(X = x) \log \Pr(X=x).
\]
Further, given two random variables $X$ and $Y$ that take values in finite sets $\cX$ and $\cY$, respectively,
and have a joint distribution, the (conditional) entropy of $X$ conditioned on $Y$ is the quantity $H(X \mid Y)$ defined as follows:
\[
  H(X \mid Y) \coloneqq \sum_{y \in \cY} \Pr(Y = y) H\big(X^{\{Y=y\}}\big),
\]
where $X^{\{Y=y\}}$ denotes $X$ conditioned on the event that $Y=y$, so that, for every $x \in \cX$
and every $y \in \cY$ with $\Pr(Y=y) \neq 0$,
\[
  \Pr\big(X^{\{Y=y\}} = x\big) = \frac{\Pr(X = x,\, Y=y)}{\Pr(Y=y)}.
\]
The above definitions ensure that entropies and conditional entropies are always nonnegative.  Moreover,
it is easy to verify that
\begin{equation}
  \label{eq:chain-rule}
  H(X \mid Y) = H(X, Y) - H(Y).
\end{equation}

In the remainder of this section, a discrete random variable will mean a random variable taking values
in some finite set.  The following elementary inequalities should be familiar to readers who have
encountered the notion of entropy.
\begin{lemma}
  \label{lem:entropy-basic}
  Suppose that $X$, $Y$, and $Z$ are discrete random variables and that $X$ takes values
  in a finite set $\cX$. We have:
  \begin{enumerate}[label=(\roman*)]
  \item
    \label{item:entropy-uniform}
    $H(X) \le \log|\cX|$ and equality holds iff $X$ is uniform on $\cX$;
  \item
    \label{item:entropy-subadditivity}
    $H(X \mid Y) \le H(X)$ and equality holds iff $X$ and $Y$ are independent;
  \item
    \label{item:entropy-more-conditioning}
    $H(X \mid Y,Z) \le H(X \mid Y)$;
  \item
    \label{item:conditional-entropy-subadditivity}
    $H(X,Y \mid Z) \le H(X \mid Z) + H(Y \mid Z)$.
  \end{enumerate}
\end{lemma}

The main ingredient in our proof is Pinsker's inequality (see \cite[Problem~3.18]{CsiKor11}), which, in our context, can be viewed as
a~`stability' version of~\ref{item:entropy-subadditivity} in Lemma~\ref{lem:entropy-basic}.
The statement requires the following notation: For two random variables $X$ and $Y$, we denote by
$X \times Y$ the random variable obtained by first letting $\tilde{X}$ and $\tilde{Y}$ be two
independent copies of $X$ and $Y$, respectively, and then defining $X \times Y = (\tilde{X}, \tilde{Y})$. In other words,
$\cL(X \times Y) = \cL(X) \times \cL(Y)$, where, as usual, $\cL(X)$ stands for the law of $X$, i.e., the measure induced by $X$ on its space of values.

\begin{lemma}
  \label{lem:Pinsker}
  Suppose that $X$ and $Y$ are discrete random variables. We have
  \[
    \dtv\big((X,Y), X \times Y\big) \le \sqrt{2 \big(H(X) - H(X \mid Y)\big)},
  \]
  where $\dtv$ denotes the total variation distance.
\end{lemma}

\subsection{The argument}
\label{sec:argument-triangles}

Let $Y$ denote the random graph $G(n, \frac{1}{2})$ conditioned on having at most
$t$ triangles.  In other words, $Y$ is a uniformly chosen
random graph with vertex set $\br{n} \coloneqq \{1, \dotsc, n\}$ and at most
$t$ triangles.  In particular, Lemma~\ref{lem:entropy-basic}\ref{item:entropy-uniform} implies that
\begin{equation}
  \label{eq:logPr-entropy}
  \log \Pr(X_n \le t) = H(Y) - \binom{n}{2} \log 2.
\end{equation}

In order to bound the entropy of $Y$ from above, it will be convenient to view $Y$ as the random vector
$(Y_e)_{e \in K_n}$, where $Y_e$ indicates whether $e$ is an edge of $Y$.  For a subvector $S$ of $Y$
and every $e \in K_n$, we will write $Y_e^S$ to denote the random variable whose (random) distribution
is the distribution of $Y_e$ conditioned on $S$, so that $\Pr(Y_e^S = 1) = \Ex[Y_e \mid S]$. The following
lemma captures the notion of \emph{conditional approximate independence} (recall the proof sketch in \S\ref{sec:triangles-Gn12}).

\begin{lemma}
  \label{claim:triangles}
  There exists a subgraph $F \subseteq K_n$ with at most $n^{3/2}$ edges and such that, for every
  $\{i, j, k\} \in \binom{\br{n}}{3}$, letting $S \coloneqq (Y_f)_{f \in F}$ and
  \[
    d_{ijk}^S \coloneqq \dtv\big((Y_{ij}^S, Y_{ik}^S, Y_{jk}^S), Y_{ij}^S \times Y_{ik}^S \times Y_{jk}^S\big),
  \]
  we have $\Ex[d_{ijk}^S] \le 2n^{-1/4}$.
\end{lemma}

\begin{proof}
  For a nonnegative integer $r$, let $F_r$ be the subgraph of $K_n$ comprising all edges $\{i,j\}$
  satisfying $\max\{i,j\}  > n-r$, let $S_r \coloneqq (Y_f)_{f \in F_r}$ and let $h_r \coloneqq H(Y_{12} \mid S_r)$. By Lemma~\ref{lem:entropy-basic}\ref{item:entropy-more-conditioning}, 
  the function $r \mapsto h_r$ is decreasing and hence, for some $r\le\sqrt{n}$, we must have
  \[
  h_r-h_{r+1}\le \frac{h_0-h_{\sqrt{n}}}{\sqrt{n}}.
  \]
  Bounding the numerator is easy. 
  On the one hand, we have
  \[
    h_0 = H(Y_{12}) \le \log 2,
  \]
  as $Y_{12}$ takes only two values, see Lemma~\ref{lem:entropy-basic}\ref{item:entropy-uniform};
  on the other hand, $h_r \ge 0$ for every $r$, as conditional entropy is always nonnegative. 
  Thus, there must be an $r$ with $0 \le r \le \sqrt{n}-1$ such that $h_r - h_{r+1} \le (\log 2)/ \sqrt{n}$.
  Fix one such $r$ and let $F = F_r$ and $S = S_r$; note that $e(F) \le rn \le n^{3/2}$.
  Since $F \subseteq F \cup \big\{\{1, n-r\}, \{2, n-r\}\big\} \subseteq F_{r+1}$, Lemma~\ref{lem:entropy-basic}\ref{item:entropy-more-conditioning} implies that
  \[
    h_{r+1} = H(Y_{12} \mid S_{r+1}) \le H(Y_{12} \mid S, Y_{1,n-r}, Y_{2,n-r}) \le H(Y_{12} \mid S) = h_r
  \]
  and, consequently,
  \begin{equation}
    \label{eq:conditional-entropy-change}
    H(Y_{12} \mid S) - H(Y_{12} \mid S, Y_{1,n-r}, Y_{2,n-r}) \le (\log 2) / \sqrt{n}.
  \end{equation}
  By symmetry (every permutation of $\br{n-r}$ fixes $F$), we may replace the triple of indices $(1, 2, n-r)$
  in~\eqref{eq:conditional-entropy-change} with any ordered triple $(i, j, k)$ of distinct
  elements of $\br{n-r}$. Using the definition of conditional entropy, we may rewrite this upgraded
  inequality as
  \[
    \Ex\Big[\underbrace{H(Y_{ij}^S) - H(Y_{ij}^S \mid Y_{ik}^S, Y_{jk}^S)}_{\lambda_{ijk}^S}\Big]
    \le (\log 2) / \sqrt{n},
  \]
  where $\Ex$ averages over the values of $S$.

  Fix an arbitrary triple $\{i, j, k\} \in \binom{\br{n}}{3}$. If $\max\{i, j, k\} > n-r$, then at least two
  out of the three pairs $ij$, $ik$, $jk$ belong to $F$; consequently, at least two out the three
  corresponding variables $Y_{ij}^S$, $Y_{ik}^S$, $Y_{jk}^S$ are trivial (for every evaluation of $S$),
  which implies that $d_{ijk}^S = 0$.  Therefore, we may assume that $\{i, j, k\} \in \binom{\br{n-r}}{3}$.
  For brevity, denote $A = Y_{ij}^S$, $B = Y_{ik}^S$, and $C = Y_{jk}^S$, so that
  \begin{align*}
      d_{ijk}^S = \dtv\big((A,B,C), A \times B \times C\big)
      & \le \dtv\big((A,B,C), A \times (B,C)\big) + \dtv\big(A \times (B,C), A \times B \times C\big) \\
      & = \underbrace{\dtv\big((A,B,C), A \times (B,C)\big)}_{d_1}
      + \underbrace{\dtv\big((B,C), B \times C\big)}_{d_2}.\\[-1.1cm]
  \end{align*}
  Pinsker's inequality (Lemma~\ref{lem:Pinsker}) implies that
  \[
    d_1 \le \sqrt{\tfrac{1}{2}\big(H(A) - H(A \mid B, C)\big)} = \sqrt{\tfrac{1}{2} \lambda_{ijk}^S}.
  \]
  Further,
  \[
    d_2 \le \dtv\big((A,B,C), (A,B) \times C\big) \le \sqrt{\tfrac{1}{2}\big(H(C) - H(C \mid A, B)\big)} = \sqrt{\tfrac{1}{2}\lambda_{jki}^S}
  \]
  (the first inequality is easy to check). We conclude that
  \[
    \begin{split}
      \Ex\big[d_{ijk}^S\big]
      & \le \Ex\left[\sqrt{\tfrac{1}{2} \lambda_{ijk}^S}\right] +  \Ex\left[\sqrt{\tfrac{1}{2} \lambda_{jki}^S}\right]
      \le \sqrt{\tfrac{1}{2} \Ex[\lambda_{ijk}^S]} + \sqrt{\tfrac{1}{2} \Ex[\lambda_{ijk}^S]} \\
      & \le 2 \sqrt{(\log 2)/(2\sqrt{n})} \le 2n^{-1/4},
  \end{split}
  \]
  where the second inequality follows from the Cauchy--Schwarz inequality.
\end{proof}

Let $F$ be the graph from the statement of Lemma~\ref{claim:triangles} and let $S = (Y_f)_{f \in F}$,
as in the claim.  The chain rule for conditional entropies, identity~\eqref{eq:chain-rule} above,
and Lemma~\ref{lem:entropy-basic}\ref{item:conditional-entropy-subadditivity} imply that
\[
  H(Y) = H(S) + H\big((Y_e)_{e \in K_n \setminus F} \mid S \big) \le H(S) + \sum_{e \in K_n \setminus F} H(Y_e \mid S).
\]
Since $S$ takes at most $2^{e(F)}$ different values and $e(F) \le n^{3/2}$, we further have,
by Lemma~\ref{lem:entropy-basic}\ref{item:entropy-uniform},
\begin{equation}
  \label{eq:HY-subadditivity}
  H(Y) \le n^{3/2} \log 2 + \sum_{e \in K_n} H(Y_e \mid S),
\end{equation}
where we also used the fact that conditional entropies are nonnegative to extend the range
of the sum from $K_n \setminus F$ to $K_n$ (in fact, $H(Y_e \mid S) = 0$ for every $e \in F$).

Recall that our eventual goal is to compare the entropy of $Y$ to $\Phi_n(t)$, which is defined as the minimum over certain functions $q$. Define therefore the $S$-measurable random function $q \colon \binom{\br{n}}{2} \to [0,1]$ by letting,
for each $e \in K_n$,
\[
  q_e \coloneqq \Ex[Y_e \mid S] = \Pr(Y_e^S = 1).
\]
Letting $h \colon [0,1] \to [0, \log 2]$ be the function defined by $h(x) = -x\log x - (1-x)\log(1-x)$, 
we may now write
\[
  H(Y_e \mid S) = \Ex\left[H(Y_e^S)\right] = \Ex[h(q_e)].
\]

Let $X^S$ denote the number of triangles in $Y$ conditioned on $S$, that is,
\[
  X^S \coloneqq \sum_{\{i,j,k\} \in \binom{\br{n}}{3}} Y_{ij}^S Y_{ik}^S Y_{jk}^S
\]
and let
\[
  \bar{X}^S \coloneqq \sum_{\{i,j,k\} \in \binom{\br{n}}{3}} q_{ij} q_{ik} q_{jk} = \Ex\left[N_{K_3}\big(G(n,q)\big)\right].
\]
Recall the definition of $d_{ijk}^S$ from the statement of Lemma~\ref{claim:triangles} and observe that
\begin{equation}
  \label{eq:dtv-triangle-ijk}
  \left| \Ex\big[Y_{ij}^S Y_{ik}^S Y_{jk}^S \big] - q_{ij} q_{ik} q_{jk}\right| \le d_{ijk}^S;
\end{equation}
indeed, the two terms in the left-hand side are the probabilities of the event that $Y_{ij}^S = Y_{ik}^S = Y_{jk}^S=1$ under the two distributions whose total variation distance is $d_{ijk}^S$.

Let
\[
  \Delta = \sum_{\{i,j,k\} \in \binom{\br{n}}{3}} d_{ijk}^S
\]
and note that, by Lemma \ref{claim:triangles},
\begin{equation}
  \label{eq:Ex-Delta-upper}
  \Ex[\Delta] \le \binom{n}{3} \cdot 2n^{-1/4} \le n^{11/4}.
\end{equation}
Summing~\eqref{eq:dtv-triangle-ijk} over all triples $\{i,j,k\}$, we obtain
\[
  \bar{X}^S \le \Ex\big[X^S\big] + \Delta \le t + \Delta,
\]
since $X^S \le t$ with probability one. In particular, the definition of $\Phi_n$, see~\eqref{eq:Phi_n}, implies that
\[
  \sum_{e \in K_n} \big(\log 2 - h(e_q)\big) \ge \Phi_n(t + \Delta).
\]
We may conclude that
\[
  \sum_{e \in K_n} H(Y_e \mid S) = \Ex\left[\sum_{e \in K_n} h(q_e) \right] \le \binom{n}{2} \log 2 - \Ex[\Phi_n(t + \Delta)].
\]
Since $\Phi_n$ is decreasing and nonnegative,
\[
  \Ex[\Phi_n(t + \Delta)] \ge \Pr(\Delta \le n^{23/8}) \cdot \Phi_n(t +n^{23/8}) \stackrel{\eqref{eq:Ex-Delta-upper}}{\ge} \big(1-n^{-1/8}\big) \cdot \Phi_n(t + n^{23/8}).
\]
Recalling~\eqref{eq:logPr-entropy} and~\eqref{eq:HY-subadditivity}, this implies that
\[
  \log\Pr(X_n \le t) \le -(1-n^{-1/8}) \cdot \Phi_n(t+n^{23/8}) + n^{3/2} \le - \Phi_n(t+n^{23/8}) + 2n^{15/8},
\]
as $\Phi_n(t) \le \binom{n}{2} \log 2$ for every $t$. This finishes the proof of Theorem~\ref{thm:triangles}.\qed

\section{The Kullback--Leibler divergence}
\label{sec:DKL}

For the proof of Theorem~\ref{thm:main-hypergraphs}, we need the notion of Kullback--Leibler divergence, or relative entropy. Let $P$ and $Q$ be random variables taking values in a finite set $\cX$ and suppose that $\cL(P) \ll \cL(Q)$, that is, that the distribution of $P$ is absolutely continuous with respect to the distribution of $Q$. Denoting by $p$ and $q$ the densities of $P$ and $Q$, respectively, the \emph{Kullback--Leibler divergence} of $P$ from $Q$ (also known as the \emph{relative entropy}), denoted by $\DKL(P \, \| \, Q)$, is defined as follows:
\[
  \DKL(P \, \| \, Q) \coloneqq \sum_{x \in \cX} p(x) \log \frac{p(x)}{q(x)},
\]
where we adopt the convention that $0 \log \frac{0}{q} = 0$ for all $q$. The assumption that $\cL(P) \ll \cL(Q)$, which is a concise way of saying that $p(x) = 0$ whenever $q(x) = 0$, guarantees that $\DKL(P \, \| \, Q)$ is well-defined. A fundamental property of the KL-divergence is that it is always nonnegative; indeed, since $\log x \le x - 1$ for all positive $x$, we have, letting $\cX' = \{x \in \cX : p(x) > 0\}$,
\[
  \DKL(P \, \| \, Q) = - \sum_{x \in \cX'} p(x) \log\frac{q(x)}{p(x)} \ge \sum_{x \in \cX'} \big( p(x) - q(x) \big) = 1 - \sum_{x \in \cX'} q(x) \ge 0.
\]

One easily checks that, when $Q$ is a uniformly chosen random element of $\cX$, then
\[
  \DKL(P \, \| \, Q) = \log |\cX| - H(P),
\]
where $H(P)$ is the entropy of $P$ (defined in the previous section). In particular, if $P$ is the uniformly chosen random element of a nonempty subset $\cA \subseteq \cX$, then, by Lemma~\ref{lem:entropy-basic}\ref{item:entropy-uniform},
\[
  \DKL(P \, \| \, Q) = \log |\cX| - \log |\cA| = - \log \Pr(Q \in \cA).
\]
The following property of the KL-divergence, which generalises this identity, is the beginning of our approach.

\begin{prop}\label{prop:P=I}
  Suppose that $Q$ is a random variable taking values in a finite set $\cX$. Suppose that $A \subseteq \cX$ satisfies $\Pr(Q \in A) \neq 0$ and let $Q^{A}$ be the random variable $Q$ conditioned on the event $\{Q \in A\}$. Then
  \[
    \DKL(Q^{A} \, \| \, Q) = - \log \Pr(Q \in A).
  \]
\end{prop}
\begin{proof}
  Let $q \colon \cX \to [0,1]$ be the probability density function of $Q$ and note that the probability density function of $Q^{A}$ is the function $q^{A} \colon \cX \to [0,1]$ defined by
  \[
    q^{A}(x) \coloneqq
    \begin{cases}
      \frac{q(x)}{\Pr(Q \in A)} & \text{if $x \in A$}, \\
      0 & \text{otherwise}.
    \end{cases}
  \]
  It follows that
  \[
    \DKL( Q^{A} \, \| \, Q) = \sum_{x \in \cX} q^{A}(x) \log \frac{q^{A}(x)}{q(x)} = \sum_{x \in A} \frac{q(x)}{\Pr(Q \in A)} \log \frac{1}{\Pr(Q \in A)} = \log \frac{1}{\Pr(Q \in A)},
  \]
  as claimed.
\end{proof}

The next property of the KL-divergence is a generalisation of the chain rule for entropies, identity~\eqref{eq:chain-rule}, and Lemma~\ref{lem:entropy-basic}\ref{item:entropy-subadditivity}.  In fact, the equality in~\eqref{eq:DKL-conditional} below is a special case of an even more general identity, the chain rule for relative entropies, see~\cite[Theorem~2.5.3]{CovTho06}.

\begin{prop}
  \label{prop:DKL-properties}
  Let $Q_1$ and $Q_2$ be random variables taking values in finite sets $\cX_1$ and $\cX_2$, respectively. Suppose that $(P_1, P_2)$ is an $\cX_1 \times \cX_2$-valued random variable such that $\cL(P_i) \ll \cL(Q_i)$ for each $i$. Let $Q_1 \times Q_2$ denote a random variable whose independent coordinates have marginals $Q_1$ and $Q_2$, respectively; that is, $\cL(Q_1 \times Q_2) = \cL(Q_1) \times \cL(Q_2)$. Then   \begin{equation}
    \label{eq:DKL-conditional}
    \DKL\big((P_1, P_2) \, \| \, Q_1 \times Q_2\big) - \DKL( P_2 \, \| \, Q_2 )
    \ge \DKL(P_1 \, \| \, Q_1),
  \end{equation}
  where equality holds if and only if $P_1$ and $P_2$ are independent.
\end{prop}

The proof of the proposition employs the following elementary inequality, whose proof we include for the sake of completeness.

\begin{lemma}
  \label{lemma:log-sum}
  Suppose that $I$ is a finite set and, for each $i \in I$, let $a_i$ and $b_i$ be nonnegative reals such that $a_i = 0$ whenever $b_i = 0$. Then, letting $a = \sum_{i \in I} a_i$ and $b = \sum_{i \in I} b_i$, we have
  \[
    \sum_{i \in I} a_i \log \frac{a_i}{b_i} \ge a \log \frac{a}{b}.  
  \]
  Moreover, equality holds above if and only if $a_ib = ab_i$ for every $i \in I$.
\end{lemma}
\begin{proof}
  Without loss of generality, we may assume that $a_i > 0$ (and thus $b_i > 0$) for each $i \in I$.
  Since the function $x \mapsto -\log x$ is strictly convex, Jensen's inequality implies that
  \[
  \sum_{i \in I} a_i \log \frac{a_i}{b_i} - a \log \frac{a}{b}
  = a \cdot \sum_{i \in I} \frac{a_i}{a} \cdot \Big(-\log \frac{a b_i}{a_ib}\Big)
  \ge -a \cdot \log \Big(\sum_{i \in I} \frac{a_i}{a} \cdot \frac{ab_i}{a_ib}\Big)
  = -a \cdot \log 1 = 0
  \]
  and the inequality is strict unless $\frac{ab_i}{a_ib} = \sum_{j \in I} \frac{a_j}{a} \cdot \frac{ab_j}{a_jb}=  1$  for every $i \in I$, as claimed.
\end{proof}

\begin{proof}[{Proof of Proposition~\ref{prop:DKL-properties}}]
  Let $p \colon \cX_1 \times \cX_2 \to [0,1]$ be the probability density function of $(P_1, P_2)$ and, for each $i \in \{1, 2\}$, let $q_i \colon \cX_i \to [0,1]$ be the probability density function of $Q_i$. Without loss of generality, we may assume that $q_i(x_i) > 0$ for every $i \in \{1,2\}$ and each $x_i \in \cX_i$. The functions $p_1 \colon \cX_1 \to [0,1]$ and $p_2 \colon \cX_2 \to [0,1]$ defined by
  \[
    p_1(x_1) \coloneqq \sum_{x_2 \in \cX_2} p(x_1, x_2)
    \qquad \text{and} \qquad
    p_2(x_2) \coloneqq \sum_{x_1 \in \cX_1} p(x_1, x_2)
  \]
  are the probability density functions of $P_1$ and $P_2$, respectively. Now, denoting by $L$ the left-hand side of~\eqref{eq:DKL-conditional}, we have
  \begin{align*}
    L &= \sum_{(x_1,x_2)\in\cX_1\times\cX_2}p(x_1,x_2)\log\frac{p(x_1,x_2)}{q_1(x_1)q_2(x_2)}
    - \sum_{x_2\in\cX_2}p_2(x_2)\log\frac{p_2(x_2)}{q_2(x_2)} \\
      &\stackrel{\mathclap{(*)}}{=}
        \sum_{x_1\in\cX_1}\sum_{x_2\in\cX_2}p(x_1,x_2)
        \log\frac{p(x_1,x_2)}{q_1(x_1)p_2(x_2)} \\
      &\stackrel{\mathclap{(\dagger)}}{\ge}
        \sum_{x_1\in\cX_1}\bigg(\sum_{x_2\in\cX_2}p(x_1,x_2)\bigg)
        \log\frac{\sum_{x_2 \in \cX_2} p(x_1, x_2)}{\sum_{x_2 \in \cX_2} q_1(x_1)p_2(x_2)} \\
      & = \sum_{x_1 \in \cX_1} p_1(x_1) \log \frac{p_1(x_1)}{q_1(x_1)} = \DKL(P_1 \, \| \, Q_1),
  \end{align*}
  where $(*)$ follows by applying $p_2(x_2)=\sum_{x_1 \in \cX_1}p(x_1,x_2)$ to the second sum and $(\dagger)$ follows by applying Lemma~\ref{lemma:log-sum} to the inner sum.
\end{proof}

\subsection{Divergence from a vector of i.i.d.\ Bernoulli variables}

Throughout this paper, we shall be estimating divergences of random variables from vectors of independent $\Ber(p)$ random variables. In view of this, it will be convenient for us to define, for a real $p \in (0,1)$, an integer $k \ge 1$, and a random variable $X$ taking values in $\{0, 1\}^k$, the \emph{$p$-divergence} $I_p(X)$ of $X$ by
\begin{equation}
  \label{eq:Ip-def-nonneg}
  I_p(X) \coloneqq \DKL\big( X \, \| \, \Ber(p)^k\big) \ge 0.
\end{equation}
When $X$ is Bernoulli itself, say with parameter $q$, then $I_p(X)$ is a function of $q$ which we will denote by $i_p$. Namely,
\begin{equation}\label{eq:defi}
  i_p(q) \coloneqq I_p\big(\Ber(q)\big) = \DKL\big(\Ber(q) \, \| \, \Ber(p)\big) = q \log\frac{q}{p} + (1-q) \log \frac{1-q}{1-p}.
\end{equation}
Let us record here, for future reference, that, for every $q \in (0,1)$,
\begin{equation}
  \label{eq:ip-derivatives}
  i_p'(q) = \log \frac{q}{p} - \log \frac{1-q}{1-p} \qquad \text{and} \qquad i_p''(q) = \frac{1}{q} + \frac{1}{1-q}.
\end{equation}
We also define a notion of conditional divergence. Given random variables $X$ and $Y$ that have a joint distribution and such that $X$ takes values in $\{0,1\}^k$ for some integer $k \ge 1$, we define the \emph{conditional $p$-divergence} of $X$ conditioned on $Y$
\[
  I_p(X \mid Y) \coloneqq \Ex\left[I_p\big(X^Y\big)\right] = \Ex\left[ \DKL\big(X^Y \, \| \, \Ber(p)^k\big) \right],
\]
where $X^Y$ denotes the random variable $X$ conditioned on $Y$, cf.\ the definition of conditional entropy.

It is straightforward to verify that, when $X$ takes values in $\{0,1\}^k$,
\begin{equation}
  \label{eq:divergence-entropy}
  I_{1/2}(X) = k \log 2 - H(X)
  \qquad
  \text{and}
  \qquad
  I_{1/2}(X \mid Y) = k\log 2 - H(X \mid Y)
\end{equation}
and therefore it should not come at a surprise that the divergence and the conditional divergence defined above
satisfy similar inequalities as entropy and conditional entropy, such as the ones presented in Lemma~\ref{lem:entropy-basic}, only in reverse.
In particular, Proposition~\ref{prop:DKL-properties} implies that\footnote{In order to see this, observe first that $I_p(X \mid Y) = \DKL\big((X,Y) \, \| \, \Ber(p)^k \times Y\big)$.}
\begin{equation}
  \label{eq:p-divergence-conditioning}
  I_p(X \mid Y) \ge I_p(X)
\end{equation}
and equality holds if and only if $X$ and $Y$ are independent, cf.~Lemma~\ref{lem:entropy-basic}\ref{item:entropy-subadditivity}; moreover, if $Y$ also takes values in $\{0,1\}^\ell$ for some integer $\ell$, then
\begin{equation}
  \label{eq:p-divergence-subadditivity}
  I_p(X,Y) = I_p(X \mid Y) + I_p(Y) \ge I_p(X) + I_p(Y),
\end{equation}
where, again, equality holds if and only if $X$ and $Y$ are independent, cf.~the chain rule for entropies (identity~\eqref{eq:chain-rule}). Generalising this further, if $Z$ is another random variable (defined on the same probability space as $X$ and~$Y$), then invoking the above inequality with $X$ and $Y$ replaced by $X^Z$ and $Y^Z$ and taking the expectation of both sides yields
\begin{equation}
  \label{eq:conditional-p-divergence-subadditivity}
  I_p(X,Y \mid Z) \ge I_p(X \mid Z) + I_p(Y \mid Z),
\end{equation}
cf.~Lemma~\ref{lem:entropy-basic}\ref{item:conditional-entropy-subadditivity}. One final property that we shall require is the following fact.
\begin{prop}
  \label{prop:p-divergence-double-conditioning}
  Suppose that random variables $X$, $Y$, and $Z$ have a joint  distribution and that $X$ takes values in $\{0,1\}^k$ for some integer $k \ge 1$. Then, for every $p \in (0,1)$,
  \[
    I_p(X \mid Y, Z) = \Ex\left[I_p(X^Y \mid Z^Y)\right]
  \]
\end{prop}
\begin{proof}
  The assertion follows from the definition of conditional $p$-divergence and the fact that
  \[
    \cL\big(X^{(Y,Z)}\big) = \cL\big((X^Y)^{Z^Y}\big)
  \]
  almost surely. 
\end{proof}

\subsection{Interlude}
\label{sec:interlude}

As an illustration of the subadditivity property of the divergence $I_p$, we will give a short proof of optimal
tail estimates for the binomial distribution (see~\cite{Csi84} for generalisations).

\begin{thm}
  For every positive integer $n$, every $p \in (0,1)$, and all $q \in [0,p]$,
  \[
    \Pr\big(\Bin(n,p) \le nq\big) \le \exp\big(-n \cdot i_p(q)\big) = \exp\left(- n \cdot \DKL\big(\Ber(q) \, \| \, \Ber(p)\big)\right).
  \]
\end{thm}
\begin{proof}
  Let $Y = (Y_1, \dotsc, Y_n)$ be a sequence of i.i.d.\ $\Ber(p)$ random variables, let $\cA$ denote the event that
  $Y_1 + \dotsb + Y_n \le nq$, and let $Y' = (Y_1', \dotsc, Y_n')$ be $Y$ conditioned on $\cA$.
  By Proposition~\ref{prop:DKL-properties},
  \[
    - \log \Pr\big(\Bin(n,p) \le nq \big) = - \log \Pr(\cA) = \DKL(Y' \,\|\, Y) = I_p(Y') \stackrel{\eqref{eq:p-divergence-subadditivity}}{\ge} \sum_{k=1}^n I_p(Y_k').
  \]
  By symmetry, for every $k \in \br{n}$,
  \[
    \Ex[Y_k'] = \frac{1}{n} \sum_{j=1}^n \Ex[Y_j'] \le q.
  \]
  In particular, since $i_p$ is decreasing on $[0,p]$ and $q\le p$, we have
  \[
    I_p(Y_k') = i_p\big(\Ex[Y_k']\big) \ge i_p(q),
  \]
  which concludes the proof of the theorem.
\end{proof}

\subsection{The key lemma}
\label{sec:key-lemma}

The following is our key lemma. Its role in the proof of Theorem~\ref{thm:main-hypergraphs} will be analogous to the role that Pinsker's inequality (Lemma~\ref{lem:Pinsker}) played in the proof of Theorem~\ref{thm:triangles}.

\begin{lemma}
  \label{lemma:main}
  Let $Y$ be a $\{0,1\}$-valued random variable and let $E_1, \dotsc, E_m$ be a sequence of $Z$-measurable events, for some random variable $Z$. Suppose that $\Ex[Y \mid Z] \le p'$ for some $p' > 0$. Then, letting $\mu = \Ex[Y] = \Pr(Y=1)$,
  \begin{equation}
    \label{eq:lemma-main}
    I_p(Y \mid Z) - I_p(Y) \ge \frac{1}{2p'} \sum_{i=1}^m \big(\Pr(Y=1 \mid E_i) - \mu\big)^2\Pr(E_i) - \frac{p'}{2} \sum_{1 \le i < j \le m}\Pr(E_i \cap E_j).
  \end{equation}
\end{lemma}

Let us first show that Lemma~\ref{lemma:main} generalises Pinsker's inequality (Lemma~\ref{lem:Pinsker}) for $\{0,1\}$-valued random variables. More precisely, let $Y \in \{0,1\}$ and $Z$ be two random variables and let $Y \times Z$ be the random variable whose independent coordinates have marginals $Y$ and $Z$.
Let $E_1$ be the $Z$-measurable event $\big\{\Pr(Y = 1 \mid Z) \le \Pr(Y = 1)\big\}$ and let $E_2$ be the complementary event. As $\Pr(Y = 1 \mid Z) - \Pr(Y = 1)$ is nonpositive on $E_1$ (respectively, nonnegative on $E_2$), we have
\[
  d_{TV}\big((Y,Z), Y \times Z\big) = \sum_{i=1}^2 (-1)^i \cdot \big(\Pr(Y = 1 \mid E_i) - \Pr(Y=1)\big) \cdot \Pr(E_i).
\]
In particular, the Cauchy--Schwarz Inequality gives
\[
  d_{TV}\big((Y,Z), Y \times Z\big)^2 \le \left(\sum_{i=1}^2 \big(\Pr(Y = 1 \mid E_i) - \Pr(Y=1)\big)^2 \cdot \Pr(E_i)\right) \cdot \big(\Pr(E_1) + \Pr(E_2)\big).
\]
It thus follows from Lemma~\ref{lemma:main}, invoked with $p = 1/2$ and $p' = 1$, that
\begin{equation}
  \label{eq:lemma-Pinsker}
  d_{TV}\big((Y,Z), Y \times Z\big)^2 \le 2 \cdot \big( I_{1/2}(Y \mid Z) - I_{1/2}(Y) \big) \stackrel{\eqref{eq:divergence-entropy}}{=} 2 \cdot \big( H(Y) - H(Y \mid Z) \big);
\end{equation}
this is precisely Pinsker's inequality (Lemma~\ref{lem:Pinsker}). In the proof of Theorem~\ref{thm:main-hypergraphs}, we will use Lemma~\ref{lemma:main} with $p' = p$, which will result in a much stronger bound.

\begin{proof}[Proof of Lemma~\ref{lemma:main}]
  Observe first that the case $\mu = 0$ is trivial. Indeed, by~\eqref{eq:p-divergence-conditioning}, the left-hand side of~\eqref{eq:lemma-main} is always nonnegative and, when $\mu = 0$, each term in the first sum in the right-hand side of~\eqref{eq:lemma-main} vanishes, as $Y=0$ almost surely. We will thus assume that $\mu > 0$. For the sake of brevity, let $g \coloneqq \Ex[Y \mid Z]$, so that
  \[
    I_p(Y \mid Z) = \Ex\big[I_p(Y^Z)\big] = \Ex[i_p(g)],
  \]
  where $i_p$ is the function defined in~\eqref{eq:defi}. Expanding $i_p$ into a Taylor series
  of order two around $\mu$ with Lagrange remainder gives
  \begin{equation}
    \label{eq:ip-Taylor}
    i_p(g) = i_p(\mu) + i_p'(\mu) \cdot (g-\mu) + i_p''(\xi_g) \cdot \frac{(g-\mu)^2}{2}
  \end{equation}
  for some $\xi_g$ with $0 < \xi_g \le \max\{\mu, g\}$. Recall from~\eqref{eq:defi} and~\eqref{eq:ip-derivatives} that the first term $i_p(\mu)$ is $I_p\big(\Ber(\mu)\big) = I_p(Y)$  and that $i_p''(\xi) = \frac{1}{\xi} + \frac{1}{1-\xi}$.
  When we take expectations (over $Z$) of both sides of~\eqref{eq:ip-Taylor}, the term $i_p'(\mu) \cdot (g - \mu)$ disappears, as $\Ex[g] = \Ex[Y] = \mu$, and thus we end up with
  \[
    I_p(Y \mid Z) - I_p(Y) = \Ex\left[\left(\frac{1}{\xi_g} + \frac{1}{1-\xi_g}\right) \cdot \frac{(g-\mu)^2}{2}\right].
  \]
  Since $\mu, g \le p'$, we have
  \[
    \frac{1}{\xi_g} + \frac{1}{1-\xi_g} \ge \frac{1}{\xi_g} \ge \min\left\{\frac{1}{\mu}, \frac{1}{g}\right\}\ge \frac{1}{p'}
  \]
  and we conclude that
  \begin{equation}
    \label{eq:Ip-diff-lower}
    I_p(Y \mid Z) - I_p(Y) \ge \frac{1}{2p'} \cdot \Ex\left[(g-\mu)^2\right] \ge \frac{1}{2p'} \int_{E_1 \cup \dotsb \cup E_m} (g-\mu)^2 \, d\Pr.
  \end{equation}
  It follows from Bonferroni's inequality (inclusion-exclusion) that
  \[
    \int_{E_1 \cup \dotsb \cup E_m} (g-\mu)^2 \,d\Pr \ge \sum_{i=1}^m \int_{E_i} (g-\mu)^2 \, d\Pr - \sum_{1\le i < j \le m} \int_{E_i \cap E_j} (g-\mu)^2 \,d\Pr.
  \]
  Since $0 \le g, \mu \le p'$, then $(g-\mu)^2 \le (p')^2$. Applying the Cauchy--Schwarz Inequality to each of the terms of the first sum above, we obtain
  \[
    \begin{split}
      \int_{E_1 \cup \dotsb \cup E_m} (g-\mu)^2 \,d\Pr &
      \ge \sum_{i=1}^m \left(\frac{1}{\Pr(E_i)}\int_{E_i} g \, d\Pr-\mu\right)^2\Pr(E_i)
      - \sum_{1\le i < j \le m} \int_{E_i \cap E_j} (p')^2 \,d\Pr \\
      & = \sum_{i=1}^m \big(\Pr(Y=1 \mid E_i) - \mu\big)^2\Pr(E_i)
      - (p')^2 \sum_{1 \le i < j \le m}\Pr(E_i \cap E_j),
    \end{split}
  \]
  which, substituted into~\eqref{eq:Ip-diff-lower}, yields the desired inequality~\eqref{eq:lemma-main}.
\end{proof}

\section{Upper bounds for the lower tail}
\label{sec:upper-bounds-lower-tail}

In this section, we prove Theorem~\ref{thm:main-hypergraphs}. Recall that we are given a hypergraph $\cH$ on a set $V$ and that $R$ denotes a random subset of $V$ where every element is included independently with probability $p$. 

\subsection{First reductions}
\label{sec:outline}

Let $Y=(Y_v)_{v \in V}$ be the indicator of $R$ conditioned on the lower tail event
$e(\cH[R]) \le \eta p^r e(\cH)$. Proposition~\ref{prop:P=I} and the definition of $I_p$ give
\[
  -\log \Pr\big(e(\cH[R])\le\eta p^re(\cH)\big)=I_p(Y),
\]
so from now on $I_p(Y)$ will be our main focus. It will be convenient to define, for every $W \subseteq V$,
\[
  H(W) \coloneqq \sum_{v \in V \setminus W} I_p\big(Y_v \mid (Y_w)_{w \in W}\big).
\]
The point of making this definition is that
\begin{align}\label{eq:19.5}
  \begin{split}
    I_p\big(Y\big) & \stackrel{\eqref{eq:p-divergence-subadditivity}}{=} I_p\big((Y_v)_{v \in V \setminus W} \mid (Y_w)_{w \in W}\big) + I_p\big((Y_w)_{w \in W}\big) \\
  & \stackrel{\eqref{eq:Ip-def-nonneg}}{\ge} I_p\big((Y_v)_{v \in V \setminus W} \mid (Y_w)_{w \in W}\big) \stackrel{\eqref{eq:conditional-p-divergence-subadditivity}}{\ge} \sum_{v \in V \setminus W} I_p\big(Y_v \mid (Y_w)_{w \in W}\big) = H(W),
\end{split}
\end{align}
and thus our goal becomes to find a set $W$ such that $H(W) \ge (1-\eps)\Phi_X(\eta+\eps) - C$.

We will relate $H(W)$ to the quantity $\Phi(\eta+\eps)$ in the following way. First, define the function $f \colon [0,1]^V \to \Reals$ by letting, for each $q \in [0,1]^V$,
\[
  f(q) \coloneqq \sum_{A \in \cH} d_A \prod_{v \in A} q_v.
\]
In other words, $f(q)$ is the expected number of edges of $\cH$ induced by a random subset of $V$ obtained by retaining each $v \in V$ independently with probability $q_v$. Note that $f(Y) = e(\cH[R])$ and that
\[
  \Phi(\eta+\eps) = \min\left\{\sum_{v \in V} i_p(q_v) : q \in [0,1]^V, f(q) \le (\eta+\eps)p^re(\cH)\right\}.
\]
Second, given a $W \subseteq V$, we define a \emph{random} function $q^W \colon V \to [0,1]$ by letting, for each $v \in V$,
\[
  q_v^W \coloneqq
  \begin{cases}
    \Ex\left[Y_v \mid (Y_w)_{w \in W}\right] & \text{if $v \notin W$},\\
    p & \text{otherwise}.
  \end{cases}
\]
Finally, we write
\[
H(W) \stackrel{(*)}{=}
\sum_{v \in V \setminus W} \Ex[i_p(q_v^W)]
= \Ex\left[\sum_{v \in V} i_p(q_v^W) \right]
\stackrel{(\dagger)}{\ge} \Pr\left(f(q^W)
\le (\eta+\eps) p^re(\cH)\right) \cdot \Phi(\eta+\eps),
\]
where $(*)$ follows from the definitions of $H$, $i_p$, and $q^W$; and where $(\dagger)$ uses $i_p \ge 0$ and bounds the expectation from below by the probability of the event $f(q^W)\le (\eta+\eps)p^re(\cH)$ times the minimum of the sum $\sum_{v \in V} i_p(q_v^W)$ on that event. 
In particular, it suffices to produce a set $W$ such that
\begin{equation}
  \label{eq:W-goal}
  \Pr\left( f(q^W) \le (\eta+\eps) p^r e(\cH)\right) \ge 1-\eps.
\end{equation}

Conditioning on $\big(Y_w\big)_{w\in W}$ for various $W \subseteq V$ will repeat so much that it is better to have a shorthand for it. Define therefore
\begin{equation}\label{eq:def EW}
  \Ex_W[\cdot]\coloneqq \Ex\left[\cdot \mid (Y_w)_{w\in W}\right]
\end{equation}
(so that our $q^W$ can now be written as $q_v^W=\Ex_W[Y_v]$ for $v \notin W$).
For similar reasons, given an $A \subseteq V$, define
\[
  Y_A \coloneqq \prod_{a\in A}Y_a.
\]

Since $f(Y) = e(\cH[R]) \le \eta p^r e(\cH)$ almost surely (and, consequently, $\Ex_W[f(Y)] \le \eta p^r e(\cH)$ for every $W \subseteq V$), we may obtain lower bounds on the probability in the left-hand side of~\eqref{eq:W-goal} by bounding from above the right-hand side of the following inequality:
\[
  \left|f(q^W) - \Ex_W[f(Y)]\right|
\le\sum_{A \in \cH} d_A \cdot \bigg| \prod_{a \in A} \Ex_W(Y_a) - \Ex_W\left[Y_A\right] \bigg|.
\]
In order to do so, we will quantify the difference between $\prod_{v \in A} \Ex_W[Y_v]$ and $\Ex_W[\prod_{v \in A} Y_v]$ for a typical $A \in \cH$. This is related to conditioned almost independence of the variables $\{Y_v\}_{v \in A}$. However, we are not studying full independence, but only with respect to the event that all $Y_v$ are $1$.
To continue our analysis, we need a few preliminaries, which will be the topic of the next section.

\subsection{Preliminaries}
\label{sec:preliminaries}

At various places we will need the following corollary of Harris's inequality:
\begin{claim}
  \label{claim:FKG}
  $\Ex_W [Y_A] \le p^{|A|}$ for all $W \subseteq V$ and all $A \subseteq V \setminus W$.
\end{claim}
\begin{proof}
  Fix some possible value $y \in \{0,1\}^W$ for $(Y_w)_{w\in W}$. Writing $E$ for the event $A \subseteq R$ and recalling that $Y$ is the indicator of $R$ conditioned on the lower tail event $e(\cH[R]) \le \eta p^r e(\cH)$,
  \begin{align*}
    \Ex\big[Y_A \mid (Y_w)_{w\in W} = y\big]
    & = \frac{\Pr\big(Y_A=1, (Y_w)_{w\in W} = y\big)}{\Pr\big((Y_w)_{w\in W}=y\big)}\\
    & = \frac{\Pr\Big(E, (R_w)_{w\in W} = y, e(\cH[R])\le\eta p^re(\cH)\Big)}
      {\Pr\Big((R_w)_{w\in W}=y, e(\cH[R]) \le\eta p^re(\cH)\Big)}\\
    & = \frac{\Pr\Big(E, e(\cH[R])\le\eta p^re(\cH) \,\Big|\, (R_w)_{w\in W}=y\Big)}
      {\Pr\Big(e(\cH[R])\le\eta p^re(\cH) \mid (R_w)_{w\in W}=y\Big)}
  \end{align*}
  Since the elements of $V$ are included in $R$ independently, conditioning on $(R_w)_{w\in W}$ gives
  a product measure on $(R_v)_{v\in V\setminus W}$. Moreover, under the conditioned measure, the event $E$
  is increasing and the lower tail event $e(\cH[R])\le\eta p^r e(\cH)$ is decreasing. The claim follows from Harris's inequality.
\end{proof}

\begin{claim}
  \label{clm:A prod a}
  For every nonempty, finite set $A$ and every function $F \colon \cP(A) \to\Reals$,
  \[
    F(A)-\prod_{a\in A}F(\{a\})
    =\sum_{\substack{B\subseteq A \\ |B| \ge 2}}\sum_{b\in B} \frac{1}{\binom{|A|}{|B|} |B|} \cdot \big(F(B)-F(B\setminus\{b\})F(\{b\})\big)
    \prod_{a\in A\setminus B}F(\{a\}).
  \]
\end{claim}
\noindent (As usual, $\cP(A)$ denotes the power set of $A$.)
\begin{proof}
  The identity holds trivially when $|A| = 1$ and we may thus assume that $|A| \ge 2$.
  Observe first that the right-hand side is a linear combination of terms of the form
  \[
    K_{\emptyset} \coloneqq \prod_{a \in A} F(\{a\})
    \qquad
    \text{and}
    \qquad
    K_B \coloneqq F(B)\cdot\prod_{a\in A\setminus B}F(\{a\}),
  \]
  where $B \subseteq A$ satisfies $|B| \ge 2$. The term $K_\emptyset$ appears only when $|B| = 2$ in the outer sum
  and it is easy to verify that its coefficient is
  \[
    - \binom{|A|}{2} \cdot 2 \cdot \frac{1}{\binom{|A|}{2} \cdot 2} = -1.
  \]
  Fix an arbitrary $B \subseteq A$ with $|B| \ge 2$. On the one hand, the term $K_B$ appears with a positive sign exactly $|B|$ times (once for each $b\in B$) and the respective coefficient is
  \[
    \frac{1}{\binom{|A|}{|B|}|B|};
  \]
  on the other hand, it appears with a negative sign ($B$ is then in fact $B \setminus \{b\}$) exactly $|A| - |B|$ times (once for each $b \in A\setminus B$) and the respective coefficient is (note that $|B| \le |A| -1$ in this case)
  \[
    \frac{-1}{\binom{|A|}{|B|+1}(|B|+1)}
  \]
  In particular, when $B \neq A$, then the positive and the negative contributions cancel, as
  \[
    |B| \cdot  \frac{1}{\binom{|A|}{|B|}|B|}= \frac{1}{\binom{|A|}{|B|}} = (|A| - |B|) \cdot \frac{1}{\binom{|A|}{|B|+1}(|B|+1)},
  \]
  and it is easy to check that the sum of the coefficients of $K_A$ is $1$.
\end{proof}

\subsection{The argument}
\label{sec:argument}

Fix an arbitrary nonempty $A \subseteq V \setminus W$. Applying Claim~\ref{clm:A prod a} with $F(B)=\Ex_W[Y_B]$ yields
\[
  \Ex_W[Y_A] - \prod_{a \in A} \Ex_W[Y_a] = \sum_{\substack{B \subseteq A \\ |B| \ge 2}} \sum_{b \in B} \frac{1}{\binom{|A|}{|B|} |B|} \cdot (\underbrace{\Ex_W[Y_B]-\Ex_W[Y_{B\setminus\{b\}}]\Ex_W[Y_b]}_{\Disc_W(B,b)}) \cdot \prod_{a \in A \setminus B} \Ex_W[Y_a]
\]
(this is the definition of $\Disc_W$). Consequently, by the triangle inequality,
\begin{align*}
  \Big| \Ex_W[Y_A] - \prod_{a \in A} \Ex_W[Y_a] \Big|
  &\le \sum_{\substack{B \subseteq A \\ |B| \ge 2}} \sum_{b \in B} \frac{1}{\binom{|A|}{|B|} |B|} \cdot \left|\Disc_W(B,b)\right| \cdot \prod_{a \in A \setminus B} \Ex_W[Y_a]\\
  & \stackrel{\mathclap{(*)}}{\le} \sum_{\substack{B \subseteq A \\ |B| \ge 2}} \sum_{b \in B} \frac{1}{\binom{|A|}{|B|} |B|} \cdot \left|\Disc_W(B,b)\right| \cdot p^{|A|-|B|},
\end{align*}
where (*) follows from Claim~\ref{claim:FKG}. We sum this inequality over all $A\in\cH-W=\cH[V\setminus W]$, take expectation over $\big(Y_v\big)_{v\in W}$, and get (recall that our hypergraph is $r$-uniform, so $|A|=r$ for every $A \in \cH$)
\begin{multline}
  \label{eq:telescope-HW}
  \Ex\left[\sum_{A \in \cH-W} d_A \cdot \Big| \Ex_W[Y_A]
      - \prod_{a \in A} \Ex_W[Y_a] \Big|\right] \\
    \le \Ex\Bigg[\sum_{A\in\cH-W}\sum_{\substack{B\subseteq A\\ |B|\ge 2}}\sum_{b\in B}
  \frac{d_A}{\binom{r}{|B|}|B|} \cdot |\Disc_W(B,b)| \cdot p^{r-|B|}\Bigg].
\end{multline}
We now wish to apply the Cauchy--Schwarz Inequality to the right-hand side of~\eqref{eq:telescope-HW}. However, since the resulting expression would be too long, we first define
\begin{equation}
  \label{eq:Err-W-def}
  \Err(W) \coloneqq \Ex\Bigg[ \sum_{A \in \cH-W} \sum_{\substack{B \subseteq A \\ |B| \ge 2}} \sum_{b \in B} \frac{d_A \cdot \Disc_W(B,b)^2}{\binom{r}{|B|}|B| \cdot p^{2|B|}} \Bigg],
\end{equation}
and then Cauchy--Schwarz yields
\begin{multline}
  \label{eq:key-inequality}
  \Ex\left[\sum_{A \in \cH-W} d_A \cdot \Big| \Ex_W[Y_A] - \prod_{a \in A} \Ex_W[Y_a] \Big|\right]  \\
  \le \bigg(\sum_{A \in \cH - W} \sum_{\substack{B \subseteq A \\ |B| \ge 2}}\sum_{b \in B} \frac{d_Ap^{2r}}{\binom{r}{|B|}|B|}\bigg)^{1/2} \cdot \Err(W)^{1/2} 
  = p^r\big((r-1)e(\cH - W)\big)^{1/2} \cdot \Err(W)^{1/2},
\end{multline}
where we used the identity $\sum_{B,b}\frac{1}{\binom{r}{|B|}|B|}=r-1$, which holds because enumerating over all $B \subseteq A$ of a given size and all $b \in B$ cancels the denominator perfectly. Let us remark that most readers might be better off ignoring all these combinatorial factors.  We chose to estimate them carefully in order to optimise the dependency of $\lambda$ and $C$ (from the statement of the theorem) on $r$. However, in most applications $r$ will be an absolute constant.

The essence of our argument is establishing the following dichotomy: Either
\begin{enumerate}[label=(\roman*)]
\item
  \label{item:dichotomy-1}
  $\Ex[\Err_W]$ is quite small, or
\item
  \label{item:dichotomy-2}
  $H(W \cup W') \ge H(W) + \Omega\big(p |V| \big)$ for some small $W' \subseteq V \setminus W$.
\end{enumerate}
If~\ref{item:dichotomy-1} holds, then, by~\eqref{eq:key-inequality}, we will have that $\Ex_W[Y_A]-\prod_{a \in A} \Ex_W[Y_a]$ is small (on average), and a few simple manipulations (done at the end of the proof of Theorem \ref{thm:main-hypergraphs}, page \pageref{pg:endofproof}) will show that our candidate set $W$ satisfies~\eqref{eq:W-goal}. Otherwise, \ref{item:dichotomy-2} holds and we replace $W$ with $W \cup W'$; this can happen only $O(1)$ times since
\begin{equation}
  \label{eq:HW-upper}
  \begin{split}
    H(W) &\stackrel{\textrm{\clap{\eqref{eq:19.5}}}}{\le} I_p(Y) = -\log \Pr\big(e(\cH[R])\le\eta p^re(\cH)\big) \\
    & \le -\log \Pr(R = \emptyset) = |V| \cdot \log\frac{1}{1-p} \le |V| \cdot \frac{p}{1-p} \le |V| \cdot \frac{p}{1-p_0}.
\end{split}
\end{equation}

\begin{lemma}
  \label{claim:main}
  For all positive $\alpha$, $\beta$, and $K$, there exist $\lambda$ and $V_0$ such that the following holds: If $|V| \ge V_0$ and $\cH$ satisfies~\eqref{eq:half} for every $s \in \br{r}$, then there exists a set $W \subseteq V$ with at most $\alpha |V|$ elements that satisfies
  \[
    \Err(W) \le \beta \cdot e(\cH).
  \]
\end{lemma}

\begin{proof} Without loss of generality, we may assume that $\alpha < 1/2$, $\beta<1$, and $K>1$. We first define a few constants:
\begin{equation}
  \label{eq:constants}
  \gamma \coloneqq \frac{\beta^2}{300Kr},
  \qquad
  \tau \coloneqq \alpha\gamma (1-p_0),
  \qquad
  \lambda \coloneqq \frac{\tau}{2r},
  \qquad
  \text{and}
  \qquad
  V_0 \coloneqq 8r^2/\tau.
\end{equation}
A short calculation shows that the definition of $V_0$ guarantees that
\begin{equation}
  \label{eq:V_0-def}
  \frac{\tau \cdot V_0/2 - r}{V_0/2 - r} \ge \frac{\tau}{2^{1/r}}
  \qquad
  \text{and}
  \qquad
  \frac{V_0/2}{V_0/2-1} \le 2^{1/(2r)}.
\end{equation}

As explained above, we shall build our set $W$ in several rounds, starting with $W$ being the empty set.  In each round, we will use the following claim, which implements the dichotomy mentioned above.
\phantom\qedhere\end{proof}

\begin{claim}
  \label{clm:WWprime}
  Suppose that $W \subseteq V$ satisfies $\Err(W) > \beta \cdot e(\cH)$ and $|V\setminus W| \ge V_0/2$. Then there exists a set $W' \subseteq V \setminus W$ with at most $\tau|V|$ elements such that
  \begin{equation}
    \label{eq:IWW'-gain}
    H(W \cup W') \ge H(W) + \gamma p|V|.
  \end{equation}
\end{claim}

\begin{proof}[Proof of Claim \ref{clm:WWprime}]
  Let $W'$ be a uniformly chosen subset of $V \setminus W$ with density $\tau$, that is, with exactly $\lfloor \tau \cdot |V \setminus W|\rfloor$ elements. 
  We will show that, under the assumption that $\Err(W) > \beta e(\cH)$ and $|V \setminus W| \ge V_0/2$, we have
  \begin{equation}
    \label{eq:H-expected-gain}
    \Ex\big[H(W \cup W')\big] \ge (1-\tau) \cdot H(W) + 2 \gamma p|V|.
  \end{equation}
  Consequently, since
  \begin{align}
    \label{eq:dontaskdonttell}
    \tau \cdot H(W) \stackrel{\eqref{eq:HW-upper}}{\le} \frac{\tau}{1-p_0} \cdot p|V| \stackrel{\eqref{eq:constants}}{=} \alpha\gamma p|V| < \gamma p |V|,
  \end{align}
  the desired inequality~\eqref{eq:IWW'-gain} must hold for some $W'$.

  We now write
  \begin{equation}
    \label{eq:29.5}
    \begin{split}
      H(W & \cup W')- H(W) =\sum_{v\in V\setminus(W\cup W')}I_p(Y_v\,|\,(Y_w)_{w\in W\cup W'}) -
      \sum_{v\in V\setminus W}I_p(Y_v\,|\,(Y_w)_{w\in W}) \\
      &=\underbrace{\sum_{v\in V\setminus(W\cup W')}I_p(Y_v\,|\,(Y_w)_{w\in W\cup W'}) - I_p(Y_v\,|\,(Y_w)_{w\in W})}_{\sumone}
      - \underbrace{\sum_{v\in W'}I_p(Y_v\,|\,(Y_w)_{w\in W})}_{\sumtwo}.
    \end{split}
  \end{equation}
  By linearity of expectation,
  \[
  \Ex\left[\sumtwo\right] = \frac{\lfloor \tau \cdot |V \setminus W|\rfloor}{|V \setminus W|} \cdot H(W) \le \tau \cdot H(W),
  \]
  and thus~\eqref{eq:H-expected-gain} will follow if we show that
  \begin{equation}
    \label{eq:sum-Jv-goal}
    \cJ \coloneqq \Ex\left[\sumone\right] \ge 2\gamma p|V|.
  \end{equation}
  
  In order to bound $\cJ$ from below, we will apply our main lemma (Lemma~\ref{lemma:main}), conditionally on $\big(Y_w:w \in W\big)$, with $Z = \big(Y_w:w \in W'\big)$ and a careful choice of the sequence of $Z$-measurable events that we shall now define. To this end, for each $v \in V \setminus W$, let
  \[
    \cH(v) \coloneqq \big\{B \subseteq V \setminus W : \text{$|B| \ge 2$, $v \in B$, and $B \subseteq A$ for some $A \in \cH - W$}\big\}
  \]
  and let $\cG(v)$ be the random subset of $\cH(v)$ formed by including each $B \in \cH(v)$ satisfying $B \setminus \{v\} \subseteq W'$ with probability $\sigma_B$, which we will specify later, independently for each such $B$.

  Let $S \coloneqq \big(Y_w\big)_{w \in W}$ and, for every $v\in V\setminus(W\cup W')$, let $Y_v^S$ denote $Y_v$ conditioned on $S$, that is, the random variable whose (random) distribution is the distribution of $Y_v$ conditioned on $S$. Define
  \[
    J^S(v) \coloneqq I_p\big(Y_v^S \mid \big(Y_w^S\big)_{w \in W'}\big) - I_p\big(Y_v^S\big),
  \]
  The next step is to apply Lemma \ref{lemma:main}. Recall that we need to supply the lemma with a sequence of events. The number of events in our application will also be random, but it will depend only on $W'$ and $\cG(v)$, so let us fix their choice for the time being. For each $B \in \cG(v)$, let $E_B^S$ be the event that $Y_{B \setminus \{v\}}^S = 1$; note that $E_B^S$ is $(Y_w^S)_{w \in W'}$-measurable, as $B \setminus \{v\} \subseteq W'$. Since $\Ex\big[Y_v^S \mid (Y_w^S)_{w \in W'}\big] = E_{W \cup W'}[Y_v] \le p$, by Claim~\ref{claim:FKG}, we may apply Lemma \ref{lemma:main} with $Y = Y_v^S$, $Z = (Y_w^S:w\in W')$, the events $E_B^S$, and $p' = p$ to get (recall the definition of $\Disc_W$ given at the start of \S~\ref{sec:argument})
  \[
    \begin{split}
      J^S(v) &\ge \frac{1}{2p}\sum_{B\in\cG(v)}\big(\Pr(Y_v^S = 1 \mid E_B^S)-\Ex[Y_v^S]\big)^2 \Pr(E_B^S)
      -\frac{p}{2}\sum_{\substack{B,B'\in\cG(v)\\B\ne B'}}\Pr(E_B^S\cap E_{B'}^S) \\
      & = \frac{1}{2p} \sum_{B \in \cG(v)} \frac{\Disc_W(B,v)^2}{\Ex_W[Y_{B\setminus\{v\}}]}
      - \frac{p}{2} \sum_{\substack{B, B' \in \cG(v) \\ B \neq B'}} \Ex_W[Y_{B \setminus \{v\}} \cdot Y_{B' \setminus \{v\}}].
    \end{split}
  \]
  Since every edge of $\cG(v)$ contains $v$ and is disjoint from $W$, Claim~\ref{claim:FKG} implies that $\Ex_W[Y_{B \setminus \{v\}}] \le p^{|B|-1}$ and $\Ex_W[Y_{B \setminus \{v\}} \cdot Y_{B' \setminus \{v\}}] \le p^{|B \cup B'|-1}$ for all $B, B' \in \cG(v)$. This observation allows us to simplify our lower bound for $J^S(v)$ to
  \begin{equation}
    \label{eq:JSv-final}
    2 \cdot J^S(v) \ge \sum_{B \in \cG(v)} \frac{\Disc_W(B,v)^2}{p^{|B|}} - \sum_{\substack{B, B' \in \cG(v) \\ B \neq B'}} p^{|B \cup B'|}\eqqcolon \gain(v)-\loss(v),
  \end{equation}
  i.e., $\gain(v)$ is the first sum and $\loss(v)$ is the second.

  We now return to the $\cJ$ from \eqref{eq:sum-Jv-goal}. It is the expectation (over $W'$) of the sum $\sumone$ defined in~\eqref{eq:29.5}, each of whose summands is the expectation (over $S$) of $J^S(v)$, see Proposition~\ref{prop:p-divergence-double-conditioning}. We wish to exchange the sum and expectation, but since the sum is over $v\not\in W'$ (recall \eqref{eq:29.5}) and this is an event, we need to condition on it. Hence we arrive at
  \begin{equation}
    \label{eq:J-gain}
    \begin{split}
      \cJ & \stackrel{\eqref{eq:sum-Jv-goal}}{=} \sum_{v \in V \setminus W} \Pr(v \notin W') \cdot \Ex\left[\Ex\big[J^S(v) \mid v \notin W'\big]\right]\\
      & \stackrel{\eqref{eq:JSv-final}}{\ge} \frac{1-\tau}{2} \cdot \Ex\left[\sum_{v \in V \setminus W} \Ex\big[\gain(v) - \loss(v) \mid v \notin W'\big]\right],
    \end{split}
  \end{equation}
  where $\Ex$ and $\Pr$ denote the expectation and the probability over the random choice of the set~$W'$ and the hypergraphs $\cG(v)$ and over $S$. In the remainder of the proof,  we shall estimate the right-hand side of~\eqref{eq:J-gain}.

  We start with the estimate of the $\gain$ terms. We define
  \[
    \gain'(v) \coloneqq \Ex\big[\gain(v) \mid v \notin W'\big] = \sum_{B \in \cH(v)} \Pr\big(B \in \cG(v) \mid v \notin W'\big) \cdot \frac{D_W(B, v)^2}{p^{|B|}}.
  \]
  For every $B \in \cH(v)$, we have (recall the assumption that $|V \setminus W| \ge V_0/2$)
  \begin{align*}
    \Pr\big(B \in \cG(v) \mid v \notin W'\big)
    & = \Pr\big(B \setminus \{v\} \subseteq W' \mid v \notin W'\big) \cdot \sigma_B \\
    & = \prod_{i=0}^{|B|-2} \frac{\lfloor\tau |V \setminus W|\rfloor - i}{|V \setminus W| - i - 1} \cdot \sigma_B \\
    & \ge \left(\frac{\tau |V \setminus W| - r}{|V \setminus W| - r}\right)^{|B|-1} \cdot \sigma_B
      \stackrel{\eqref{eq:V_0-def}}{\ge} \frac{\tau^{|B|-1}}{2} \cdot \sigma_B.
  \end{align*}
  We conclude that
  \begin{equation}
    \label{eq:Ex'-gain}
    \gain'(v) \ge \frac{1}{2\tau}\sum_{B \in \cH(v)} \frac{\tau^{|B|} \cdot \sigma_B  \cdot D_W(B, v)^2}{p^{|B|}}.
  \end{equation}
  Summing~\eqref{eq:Ex'-gain} over all $v \in V \setminus W$ yields (recall the definition of $\deg_{\cH-W}$ given in~\eqref{eq:deg-cH-B})
  \begin{equation}
    \label{eq:sum-Gv}
    \sum_{v \in V \setminus W} \gain'(v) \ge \frac{1}{2\tau} \sum_{A \in \cH-W} \sum_{\substack{B \subseteq A \\ |B| \ge 2}} \sum_{v \in B} \frac{d_A}{\deg_{\cH-W}B} \cdot \frac{\tau^{|B|} \cdot \sigma_B \cdot \Disc_W(B,v)^2}{p^{|B|}},
  \end{equation}
  cf.~the definition of $\Err(W)$ given in~\eqref{eq:Err-W-def}.
  This is a good moment to finally define the probabilities~$\sigma_B$. We let
  \begin{equation}
    \label{eq:def sigma}
    \sigma_B \coloneqq \mu \cdot \frac{\deg_{\cH-W}B}{\binom{r}{|B|}|B|(\tau p)^{|B|}},
  \end{equation}
  where
    \begin{equation}
      \label{eq:mu-def}
      \mu \coloneqq \frac{\beta \tau p |V|}{16K(r-1)e(\cH)}.
  \end{equation}
  Note that $\sigma_B \le 1$ as
  \[
  \deg_{\cH-W}B
  \le \Delta_{|B|}(\cH)
  \stackrel{\textrm{\eqref{eq:half}}}{\le}
  K \cdot (\lambda p)^{|B|-1} \cdot \frac{e(\cH)}{|V|}
  \stackrel{\eqref{eq:constants}}{\le} K \cdot (\tau p)^{|B|-1} \cdot \frac{e(\cH)}{|V|}
  \stackrel{\eqref{eq:mu-def}}{\le} \frac{(\tau p)^{|B|}}{\mu}.
  \]
  Substituting~\eqref{eq:def sigma} into~\eqref{eq:sum-Gv} yields precisely
  \begin{equation}
    \label{eq:sum-Gsv-final}
    \Ex\bigg[\sum_{v \in V \setminus W} G'(v)\bigg] \ge \frac{\mu}{2\tau}\sum_{A\in\cH-W}\sum_{\substack{B\subseteq A\\ |B|\ge 2}}\sum_{v\in B}\frac{d_A \cdot \Ex[D_W(B,v)^2]}{\binom{r}{|B|}|B|p^{2|B|}} \stackrel{\eqref{eq:Err-W-def}}{=} \frac{\mu}{2 \tau} \cdot \Err(W).
  \end{equation}
  This concludes our estimate of the $\gain$ terms.
  
  The estimate of the $\loss$ terms in~\eqref{eq:J-gain} is similar, but somewhat more involved.  We define
  \[
    \loss'(v) \coloneqq \Ex\big[\loss(v) \mid v \notin W'\big] = \sum_{\substack{B, B' \in \cH(v) \\ B \neq B'}} \Pr\big(B, B' \in \cG(v) \mid v \notin W'\big) \cdot p^{|B \cup B'|}.
  \]
  Thus, we need a second moment estimate for the sum of indicators of $B \in \cG(v)$ over all $B \in \cH(v)$. Note first that, for each $B\ne B'$,
  \[
    \begin{split}
      \Pr\big(B, B' \in \cG(v) \mid v \notin W'\big)
      & = \Pr\big((B \cup B') \setminus \{v\} \subseteq W' \mid v\notin W'\big) \cdot \sigma_B \sigma_{B'} \\
      & = \prod_{i=0}^{|B \cup B'|-2} \frac{\lfloor\tau |V \setminus W|\rfloor - i}{|V \setminus W| - i - 1} \cdot \sigma_B \sigma_{B'} \\
      & \le \left(\frac{\tau |V \setminus W|}{|V \setminus W| - 1}\right)^{|B \cup B'| - 1} \cdot \sigma_B \sigma_{B'}
      \stackrel{\eqref{eq:V_0-def}}{\le} 2 \tau^{|B \cup B'|-1} \cdot \sigma_B \sigma_{B'}.
    \end{split}
  \]
  Hence
  \begin{equation}
    \label{eq:Ex'-loss}
    \loss'(v) \le \frac{2}{\tau} \sum_{\substack{B, B' \in \cH(v) \\ B \neq B'}} (\tau p)^{|B\cup B'|} \cdot \sigma_B \sigma_{B'}.
  \end{equation}
  Summing~\eqref{eq:Ex'-loss} over all $v \in V \setminus W$ gives
  \[
    \begin{split}
      \sum_{v \in V \setminus W} \loss'(v) & \le \frac{2}{\tau} \sum_{A, A' \in \cH-W} \sum_{\substack{B \subseteq A, B' \subseteq A' \\ |B|,  |B'| \ge 2 \\ B \neq B'}} \sum_{v \in B \cap B'} \frac{d_A}{\deg_{\cH-W}B} \cdot \frac{d_{A'}}{ \deg_{\cH-W}B'} \cdot (\tau p)^{|B \cup B'|}\cdot \sigma_B \sigma_B' \\
      & \stackrel{\eqref{eq:def sigma}}{=} \frac{2\mu^2}{\tau^2p} \cdot \underbrace{\sum_{A, A' \in \cH-W} \sum_{\substack{B \subseteq A, B' \subseteq A' \\ |B|, |B'| \ge 2 \\ B \neq B'}}\frac{|B \cap B'| \cdot d_A d_{A'}}{\binom{r}{|B|}|B|\binom{r}{|B'|}|B'| (\tau p)^{|B \cap B'|-1}}}_{(*)},
    \end{split}
  \]
  where we used the identity $|B \cup B'| + |B \cap B'| = |B| + |B'|$.
  Rearranging gives
\[
  (*)=\sum_{A\in\cH-W}d_A\sum_{\substack{B\subseteq A\\ |B| \ge 2}} \frac{1}{\binom{r}{|B|}|B|}\sum_{s=1}^{r-1}
  \frac{s}{(\tau p)^{s-1}}
  \underbrace{\sum_{\substack{C\subseteq B\\|C|=s}} \sum_{\substack{A'\in\cH-W\\C\subseteq A'}} d_{A'} \sum_{\substack{B' \subseteq A' \\ |B|'\ge2 \\ B\cap B'=C}}\frac{1}{\binom{r}{|B'|}|B'|}}_{S_{B,s}}.
\]
Now, for every $A' \in \cH-W$, every $s \ge 1$, and every $C \subseteq A'$ with $|C| = s$,
\[
  \sum_{C \subseteq B' \subseteq A'} \frac{1}{\binom{r}{|B'|}|B'|} = \sum_{b' = s}^r \frac{\binom{r-s}{b'-s}}{\binom{r}{b'}b'} = \sum_{b'=s}^r \frac{\binom{b'}{s}}{\binom{r}{s}b'} =  \sum_{b'=s}^r \frac{\binom{b'-1}{s-1}}{\binom{r}{s}s} = \frac{1}{s} \le 1.
\]
Hence, for every $B$ with at most $r$ elements and every $s \ge 1$,
\[
  \begin{split}
    S_{B,s}
    & \le \sum_{\substack{C \subseteq B \\ |C| = s}} \sum_{\substack{A' \in \cH-W \\ C \subseteq A'}} d_{A'}
    \le \binom{|B|}{s} \cdot \Delta_s(\cH)
    = \frac{|B|}{s} \binom{|B|-1}{s-1} \cdot \Delta_s(\cH)\\
    & \stackrel{\eqref{eq:half}}{\le}
    \frac{|B|}{s} \binom{|B|-1}{s-1}\cdot (\lambda p)^{s-1} \cdot K \cdot \frac{e(\cH)}{|V|}
    \le \frac{|B|}{s} \cdot (r\lambda p)^{s-1} \cdot K \cdot \frac{e(\cH)}{|V|}.
  \end{split}
\]
Consequently,
\[
  \begin{split}
    (*) & \le \sum_{A \in \cH-W} d_A \sum_{\substack{B \subseteq A \\ |B| \ge 2}} \frac{1}{\binom{r}{|B|}} \sum_{s=1}^{r-1} \left(\frac{r\lambda}{\tau}\right)^{s-1} \cdot K \cdot \frac{e(\cH)}{|V|}\\
    & = e(\cH-W) \cdot (r-1) \cdot \sum_{s=1}^{r-1} \left(\frac{r\lambda}{\tau}\right)^{s-1} \cdot  K \cdot \frac{e(\cH)}{|V|}.
\end{split}
\]
Since $r \lambda = \tau / 2$, we conclude that
\begin{align}\label{eq:39.5}
  \sum_{v \in V \setminus W} L'(v) \le \frac{2\mu^2}{\tau^2p} \cdot (r-1) \cdot 2K \cdot \frac{e(\cH)^2}{|V|}
  \stackrel{\eqref{eq:mu-def}}{=} \frac{\mu}{\tau} \cdot \frac{\beta e(\cH)}{4}.
\end{align}
Combining this with the estimate~\eqref{eq:sum-Gsv-final} gives
\[
  \begin{split}
    \cJ & \stackrel{\eqref{eq:J-gain}}{\ge}
    \frac{1-\tau}{2} \cdot \Ex\left[ \sum_{v \in V \setminus W} \gain'(v) - \loss'(v) \right]
    \stackrel{\textrm{(\ref{eq:sum-Gsv-final},\ref{eq:39.5})}}{\ge}
    \frac{(1-\tau)}{2} \cdot \frac{\mu}{\tau} \cdot \left(\frac{\Err(W)}{2} -  \frac{\beta e(\cH)}{4}\right) \\
    & \stackrel{(*)}{>}
    \frac{(1-\tau)}{2} \cdot \frac{\mu}{\tau} \cdot \frac{\beta e(\cH)}{4}
    \stackrel{~\eqref{eq:mu-def}}{=}
    \frac{(1-\tau)\beta^2}{128K(r-1)} \cdot p|V|
    \stackrel{\eqref{eq:constants}}{\ge} 2\gamma p|V|,
  \end{split}
\]
where $(*)$ follows from our assumption that $\Err(W) > \beta e(\cH)$. The claim is thus proved. 
\end{proof}

\begin{proof}[Proof of Lemma \ref{claim:main}, continued]
  Suppose that the assertion of the lemma is not true, that is, $\Err(W) > \beta \cdot e(\cH)$ for every $W \subseteq V$ with at most $\alpha|V|$ elements. We will construct a sequence $W_0, \dotsc, W_j$ of subsets of~$V$, where $j = \lfloor \alpha / \tau \rfloor + 1$, such that, for each $i \in \{0, \dotsc, j\}$,
  \begin{enumerate}[label=(\roman*)]
  \item
    \label{item:Wi-size}
    $|W_i| \le i \cdot \tau |V|$ and
  \item
    \label{item:HWi-lower}
    $H(W_i) \ge i \cdot \gamma p|V|$.
  \end{enumerate}
  If such a sequence existed, we would have
  \[
  H(W_j) \ge j \cdot \gamma p|V|
  > (\alpha / \tau) \cdot \gamma p|V|
  \stackrel{\eqref{eq:constants}}{=} |V| \cdot \frac{p}{1-p_0},
  \]
  which contradicts~\eqref{eq:HW-upper}.
  
  We start by letting $W_0 = \emptyset$. Suppose that $0 \le i \le j-1$ and that $W_i$ has already been defined so that~\ref{item:Wi-size} and~\ref{item:HWi-lower} hold. Since
  \[
  |W_i| \le i \cdot \tau |V|
  \le \lfloor \alpha/\tau \rfloor \cdot \tau |V|
  \le \alpha |V|,
  \]
  we have $\Err(W_i) > \beta \cdot e(\cH)$ by the contradictory assumption. We note also that $|V \setminus W_i| \ge (1-\alpha)|V| \ge |V|/2 \ge V_0/2$. In particular, Claim~\ref{clm:WWprime}, invoked with $W = W_i$, supplies a $W' \subseteq V \setminus W_i$ with at most $\tau |V|$ elements that satisfies~\eqref{eq:IWW'-gain}. We let $W_{i+1}=W_i\cup W'$ and note that
  \[
    |W_{i+1}| = |W_i| + |W'| \stackrel{\textrm{\ref{item:Wi-size}}}{\le} i \cdot \tau |V| + \tau |V| = (i+1) \cdot \tau |V|
  \]
  and
  \[
  H(W_{i+1}) = H(W_i \cup W')
  \stackrel{\eqref{eq:IWW'-gain}}{\ge}
  H(W_i) + \gamma p |V|
  \stackrel{\textrm{\ref{item:HWi-lower}}}{\ge} (i+1) \cdot \gamma p |V|,
  \]
  so~\ref{item:Wi-size} and~\ref{item:HWi-lower} continue to hold with $i$ replaced by $i+1$. This completes the proof of the existence of the sequence of $W_0, \dotsc, W_j$, which yields the desired contradiction.
\end{proof}

\begin{proof}[Proof of Theorem \ref{thm:main-hypergraphs}]
  Let $\lambda$ and $V_0$ be constants supplied by Lemma~\ref{claim:main} invoked with
  \[
    \alpha \coloneqq \frac{\eps}{2K}
    \qquad
    \text{and}
    \qquad
    \beta \coloneqq \frac{\eps^4}{4r}.
  \]
  (We note here that $\lambda \ge 2^{-15}K^{-2}r^{-4}\eps^9(1-p_0)$ and $V_0 = 4r/\lambda$.)
  We first handle the uninteresting case $|V| < V_0$. Considering, in the definition of $\Phi(\eta)$, the function $q \colon V \to [0,1]$ that assigns zero to all elements of $V$ shows that
\[
  \Phi(\eta+\eps) \le \Phi(0) \le |V| \cdot i_p(0) = - |V| \cdot \log(1-p) \le -V_0 \cdot \log(1-p_0).
\]
In particular, setting $C \coloneqq -V_0 \cdot \log(1-p_0)$ makes the assertion of the theorem hold vacuously. 

\phantomsection\label{pg:endofproof}

We may thus assume that $|V| \ge V_0$, so that Lemma~\ref{claim:main} supplies a set $W \subseteq V$ with at most $\eps/(2K) \cdot |V|$ elements such that $\Err(W)\le \eps^4/(4r) \cdot e(\cH)$. Let $q^W \colon V \to [0,1]$ be the random function defined in the proof outline, that is, $q_v^W \coloneqq \Ex_W[Y_v]$ for $v \in V \setminus W$ and $q_v^W \coloneqq p$ for $v \in W$. We have
\[
  \Ex\left[\sum_{A \in \cH-W} d_A \cdot \Big| \Ex_W[Y_A] - \prod_{a \in A} q_a^W \Big|\right]  \stackrel{\eqref{eq:key-inequality}}{\le} p^r\big(re(\cH)\big)^{1/2} \cdot \Err(W)^{1/2} \le \frac{\eps^2}{2} \cdot p^r e(\cH).
\]
In particular, it follows from Markov's inequality that, with probability at least $1-\eps$,
\[
  \sum_{A \in \cH-W} d_A \prod_{a \in A} q_a^W \le \sum_{A \in \cH-W} d_A \cdot \Ex_W[Y_A] + \frac{\eps}{2} \cdot p^r e(\cH).
\]
However, the definition of $Y$ implies that, deterministically, 
\[
  \sum_{A \in \cH-W} d_A Y_A \le \sum_{A \in \cH} d_A Y_A = e(\cH[R]) \le \eta p^re(\cH)
\]
and thus, with probability at least $1-\eps$,
\[
  \sum_{A \in \cH-W} d_A \prod_{a \in A} q_a^W \le (\eta+\eps/2) \cdot p^re(\cH).
\]
The definition of $q^W$ and Claim~\ref{claim:FKG} guarantee that $q_v^W \le p$ for every $v \in V$ and, therefore,
\[
  \begin{split}
    \sum_{A \in \cH \setminus (\cH-W)} d_A \prod_{a \in A}q_a^W &
    \le p^r \cdot \big(e(\cH) - e(\cH-W)\big)
    \le p^r \cdot |W| \cdot \Delta_1(\cH) \\
    & \stackrel{\eqref{eq:half}}{\le}
    p^r \cdot \frac{\eps |V|}{2K} \cdot K \cdot \frac{e(\cH)}{v(\cH)}
    = \frac{\eps}{2} \cdot p^r e(\cH).
\end{split}
\]
Summarising, with probability at least $1-\eps$, we have
\[
  f(q^W) = \sum_{A \in \cH} d_A \prod_{a \in A} q_a^W \le (\eta+\eps) \cdot p^r e(\cH).
\]
Hence, we may conclude that
\[
  H(W) \stackrel{\textrm{\eqref{eq:W-goal}}}{\ge} \Pr\left(f(q^W) \le (\eta+\eps) p^re(\cH)\right) \cdot \Phi(\eta+\eps) \ge (1-\eps) \Phi(\eta+\eps),
\]
as needed.
\end{proof}

\section{Lower bounds for the lower tail}
\label{sec:lower}

In this section, we prove Theorem~\ref{thm:lower-bound}. We will need the following technical lemma.

\begin{lemma}
  \label{lem:calc Var}
  For every $p_0<1$, there exists a constant $K$ such that the following holds. Suppose that $0 < p \le p_0$ and $0 \le q \le p$, let $Y \sim \Ber(q)$, and let
  \[
    X \coloneqq Y \log\frac{q}{p} + (1-Y) \log \frac{1-q}{1-p}.
  \]
  Then,
  \[
    \Var(X) \le K \Ex[X] = K i_p(q).
  \]
\end{lemma}
\begin{proof}
  This is nothing but a calculus exercise, but let us do it in details anyway. Observe first that the case $q=0$ is trivial. Indeed, $i_p$ is nonnegative and $\Var(X) = 0$ when $q=0$. We will thus assume that $q > 0$. A direct computation shows that
  \[
    \Var(X) = q(1-q) \left( \log\frac{q}{p} - \log\frac{1-q}{1-p}\right)^2 \stackrel{\eqref{eq:ip-derivatives}}{=} q(1-q) \big(i_p'(q)\big)^2 \le q \cdot \big(i_p'(q)\big)^2.
    \]
  Since $i_p(p) = i_p'(p) = 0$, expanding both $i_p(q)$ and $i_p'(q)$ in Taylor series around $q = p$ with Lagrange remainder gives $q_1, q_2 \in (q,p)$ such that
  \begin{align}
    i_p(q) & = \frac{(q-p)^2}{2} \cdot \left(\frac{1}{q_1} + \frac{1}{1-q_1}\right) \ge \frac{(q-p)^2}{2p}, \label{eq:46.5}\\
    i_p'(q) & = (q-p) \cdot \left(\frac{1}{q_2} + \frac{1}{1-q_2}\right).\nonumber
  \end{align}
  Suppose first that $q \ge p/2$. Our assumption that $p \le p_0$ implies that
  \[
  \frac{1}{q_2} + \frac{1}{1-q_2} \le \frac{1}{q} + \frac{1}{1-p}
  \le \frac{2}{p} + \frac{1}{1-p_0}
  \le \left(2 + \frac{p_0}{1-p_0} \right) \cdot \frac{1}{p} = \frac{2-p_0}{1-p_0} \cdot \frac{1}{p}
  \]
  and, consequently,
  \[
  \Var(X) \le q \cdot \big(i_p'(q)\big)^2
  \le \left(\frac{2-p_0}{1-p_0}\right)^2\cdot \frac{q \cdot (q-p)^2}{p^2}
  \le \left(\frac{2-p_0}{1-p_0}\right)^2 \cdot 2i_p(q).
  \]
  If, on the other hand, $q < p/2$, then, using the inequality $(a-b)^2 \le 2a^2 + 2b^2$, we get
  \[
    \begin{split}
      \frac{\Var(X)}{p}
      & \le \frac{q}{p} \left( \log\frac{q}{p} - \log\frac{1-q}{1-p}\right)^2
      \le \frac{2q}{p}\left(\log\frac{q}{p}\right)^2 + \frac{2q}{p}\left(\log\frac{1-q}{1-p}\right)^2 \\
      & \le \sup_{x\in (0,1/2)} 2x(\log x)^2 + \left(\log\frac{1}{1-p}\right)^2
      \le \frac{8}{e^2} + \left(\log\frac{1}{1-p_0}\right)^2,
    \end{split}
  \]
  whereas, since $i_p$ is decreasing in the interval $[0,p]$,
  \[
    \frac{i_p(q)}{p} \ge \frac{i_p(p/2)}{p} \ge \frac{(p/2)^2}{2p^2} = \frac{1}{8}.\qedhere
  \]
\end{proof}

\begin{proof}[{Proof of Theorem~\ref{thm:lower-bound}}]
  We may assume without loss of generality that $\eps < 1$. Let $q \colon V \to [0,1]$ be the minimiser in the definition of $\Phi\big((1-\eps)\eta\big)$ and let $Y' = (Y_v')_{v \in V}$ be a sequence of independent Bernoulli random variables with $\Ex[Y_v']=q_v$ for each $v \in V$, so that
  \[
    \Ex[f(Y')] \le (1-\eps)\eta\Ex[f(Y)]
    \qquad \text{and} \qquad
    \sum_{v \in V} i_p(q_v) = \Phi\big((1-\eps)\eta\big).
  \]
  We claim that $q_v \le p$ for every $v \in V$. Indeed, otherwise $i_p(q_v) > 0 = i_p(p)$ and changing $q_v$ to $p$ can only decrease $\Ex[f(Y')]$.
  Let $\cY \subseteq \{0,1\}^V$ be arbitrary and note that
  \[
    \begin{split}
      \Pr\big(Y' \in \cY\big) & = \sum_{y \in \cY} \frac{\Pr(Y' = y)}{\Pr(Y=y)} \cdot \Pr(Y=y) = \sum_{y \in \cY} \prod_{v : y_v =1} \frac{q_v}{p} \prod_{v : y_v = 0} \frac{1-q_v}{1-p} \cdot \Pr(Y = y) \\
      & \le \max\left\{\exp\left(\sum_{v:y_v=1} \log\frac{q_v}{p} + \sum_{v:y_v=0}\log\frac{1-q_v}{1-p}\right) \colon y \in \cY \right\} \cdot \Pr(Y \in \cY).
    \end{split}
  \]
  In view of this, define, for each $y \in \{0,1\}^V$,
  \[
    J(y) \coloneqq \sum_{v:y_v=1} \log\frac{q_v}{p} + \sum_{v:y_v=0} \log\frac{1-q_v}{1-p},
  \]
  so that the above inequality may be rewritten as
  \begin{equation}
    \label{eq:RJ}
    \Pr(Y' \in \cY) \le \max_{y\in\cY}\exp\big(J(y)\big)\cdot\Pr(Y \in \cY).
  \end{equation}

  Now, let $K$ be the constant given by Lemma~\ref{lem:calc Var}, let $C' \coloneqq K/(2\eps^2)$, and define
  \begin{align}
    \nonumber
    \cY_1 & \coloneqq \left\{y \in \{0,1\}^V : f(y) \le \eta \Ex[f(Y)]\right\}, \\
    \label{eq:def cA2}
    \cY_2 & \coloneqq \left\{y \in \{0,1\}^V : J(y) \le (1+\eps) \Phi\big((1-\eps)\eta\big) + C'\right\},
  \end{align}
  It is immediate from these definitions that
  \[
    \begin{split}
      \Pr\big(X \le \eta\Ex[X]\big) & = \Pr(Y \in \cY_1) \ge \Pr(Y \in \cY_1 \cap \cY_2) \stackrel{\textrm{\eqref{eq:RJ}}}{\ge}
      \Pr(Y' \in \cY_1 \cap \cY_2) \cdot \exp\left(-\max_{y \in \cY_2} J(y)\right)\\
      & \stackrel{\textrm{\eqref{eq:def cA2}}}{\ge} \Pr(Y' \in \cY_1 \cap \cY_2)\cdot \exp\left(-(1+\eps)\Phi\big((1-\eps)\eta\big)-C'\right).
    \end{split}
  \]
  We will show that $\Pr(Y' \in \cY_1 \cap \cY_2) \ge \eps/2$, which will yield the assertion of the theorem with $C \coloneqq C' + \log(2/\eps)$.

  Since $f$ is nonnegative, Markov's inequality gives
  \[
    \Pr\big(f(Y') > \eta\Ex[f(Y)]\big) \le 1-\eps
  \]
  and thus
  \[
    \Pr(Y' \in \cY_1) = \Pr\big(f(Y') \le \eta\Ex[f(Y)]\big) \ge \eps;
  \]
  in particular, it is enough to show that $\Pr(Y' \notin \cY_2) \le \eps/2$. To this end, examine $J(Y')$.
  It is a sum of independent variables $(X_v)_{v\in V}$, where each $X_v$ is distributed exactly like the
  $X$ of Lemma~\ref{lem:calc Var}, only with $q$ replaced by $q_v$. In particular,
  \[
    \Ex[J(Y')] = \sum_{v \in V} \Ex[X_v] = \sum_{v \in V} i_p(q_v) = \Phi\big((1-\eps)\eta\big)
  \]
  and
  \[
    \Var(J(Y')) = \sum_{v \in V} \Var(X_v) \le K \sum_{v \in V} \Ex[X_v] = K \Ex[J(Y')].
  \]
  Therefore, writing $\mu \coloneqq \Ex[J(Y')]$, Chebyshev's inequality gives
  \[
    \begin{split}
      \Pr(Y' \notin \cY_2 ) & = \Pr\big(J(Y') > (1+\eps)\mu + C'\big) \le \frac{\Var(J(Y'))}{(\eps \mu + C')^2} \le \frac{K\mu}{(\eps\mu + C')^2} \\
      & \le \max_{x \ge 0} \frac{Kx}{(\eps x + C')^2} = \max_{y > 0} \frac{K}{(\eps y + C'/y)^2} = \frac{K}{4C'\eps} = \frac{\eps}{2},
    \end{split}
  \]
  as desired.
\end{proof}

\section{Applications}
\label{sec:applications}

In this section, we derive Theorems~\ref{thm:graphs}, \ref{thm:hypergraphs}, and \ref{thm:APs} from our main technical result, Theorem~\ref{thm:main-hypergraphs}, and the general lower bound estimate for lower tail probabilities, Theorem~\ref{thm:lower-bound}.  In order to do so, we just need to represent the number of copies of a given (hyper)graph $H$ in subgraphs of the complete (hyper)graph (resp.\ the number of arithmetic progressions of a given length in subsets of positive integers) as the number of edges in some auxiliary hypergraph $\cH$ and verify that $\cH$ satisfies the assumptions of Theorem~\ref{thm:main-hypergraphs} when $p \gg n^{-1/m_r(H)}$.  This is pretty straightforward, but we present the full details for the reader's convenience.

The following easy lemma, which states that $\Phi_p^{\cH}$, defined in~\eqref{eq:3.5} above the statement of Theorem~\ref{thm:main-hypergraphs}, satisfies $\Phi_p^{\cH}(\eta) = \Theta\big(v(\cH)p\big)$ for every uniform hypergraph $\cH$ whose maximum degree is comparable to its average degree, will be used to absorb the additive constant~$C$ from the assertions of Theorems~\ref{thm:main-hypergraphs} and~\ref{thm:lower-bound} into the main term.

\begin{lemma}
  \label{lemma:Phi-p-cH-lower}
  Suppose that $\cH$ is an $r$-uniform hypergraph that satisfies
  \[
    \Delta_1(\cH) \le K \cdot \frac{e(\cH)}{v(\cH)}
  \]
  for some $K$. Then, for all positive reals $p$ and $\eps$,
  \[
    \Phi_p^{\cH}(1-\eps) \ge \frac{\eps^2}{2K^2} \cdot |V|p.
  \]
\end{lemma}
\begin{proof}
  Let $q \colon V \to [0,1]$ be a function achieving the minimum in the definition of $\Phi_p^{\cH}$ and note that $q_v \le p$ for every $v \in V$.  Indeed, otherwise $i_p(q_v) > 0 = i_p(p)$ and changing $q_v$ to $p$ can only decrease $\Ex[e(\cH[R^{(q)}])]$.  As $q_v \le p$ for every $v \in V$, it is easy to conclude that 
  \begin{equation}
    \label{eq:easy}
    p^{|A|} - \prod_{v \in A} q_v \le \sum_{v \in A} (p-q_v)p^{|A|-1}
  \end{equation}
  for every $A \subseteq V$. We may thus conclude that
  \[
    \begin{split}
      \eps p^r e(\cH) &
      \stackrel{(\star)}{\le} p^r e(\cH) - \Ex[e(\cH[R^{(q)}])]
      = \sum_{A\in\cH}d_A \cdot \left(p^{|A|}-\prod_{v\in A}q_v\right) \\
      &  \stackrel{\eqref{eq:easy}}{\le}
      \sum_{A \in \cH} d_A \cdot \sum_{v \in A} (p-q_v)p^{r-1}
      = \sum_{v \in V} (p-q_v)p^{r-1} \cdot \deg_{\cH}v \\
      & \le \Delta_1(\cH) \cdot \sum_{v \in V} (p-q_v) p^{r-1}
      =p^{r-1}\Delta_1(\cH) \cdot \left( p|V|-\sum_{v\in V}q_v\right), 
    \end{split}
    \]
    where $(\star)$ follows because $q$ is the minimiser of $\Phi(1-\eps)$. Consequently,
  \begin{equation}\label{eq:48.5}
  \bar{q} \coloneqq \frac{1}{|V|} \sum_{v \in V} q_v
  \le p \cdot \left(1 - \frac{\eps e(\cH)}{|V| \cdot \Delta_1(\cH)}\right)
  \le p \cdot \left(1 - \frac{\eps}{K}\right).
  \end{equation}
  Since the function $i_p$ is convex and $i_p(q) \ge \frac{(q-p)^2}{2p}$ when $q \le p$, see~\eqref{eq:46.5}, we may conclude that
  \[
  \Phi_p^{\cH}(1-\eps)
  = \sum_{v \in V} i_p(q_v)
  \ge |V| \cdot i_p(\bar{q})
  \ge |V| \cdot \frac{(\bar{q}-p)^2}{2p}
  \stackrel{\eqref{eq:48.5}}{\ge} \frac{\eps^2}{2K^2} \cdot |V|p,
  \]
  as claimed.
\end{proof}

\begin{proof}[Proof of Theorems~\ref{thm:graphs} and~\ref{thm:hypergraphs}]
  Theorem~\ref{thm:graphs} is merely the special case $s=2$ in Theorem~\ref{thm:hypergraphs}, so we focus on Theorem~\ref{thm:hypergraphs}.  Suppose that $H$ is a nonempty $s$-uniform hypergraph and let $\cH$ be the $e_H$-uniform hypergraph with vertex set $V \coloneqq \binom{\br{n}}{s}$ whose hyperedges are the edge sets of all $\frac{v_H!}{|\Aut(H)|} \cdot \binom{n}{v_H}$ copies of $H$ in the complete $s$-uniform hypergraph on $\br{n}$ (we take $d_A=1$ for all $A$). By symmetry,
  \[
    \Delta_1(\cH) = \frac{e_H \cdot e(\cH)}{v(\cH)}.
  \]
  Suppose now that $B \subseteq V$ has at least two elements and nonzero degree in $\cH$. Then $B$ must be the edge set of some copy of a subhypergraph $F \subseteq H$, with $e_F = |B| \ge 2$, in the complete $s$-uniform hypergraph on $\br{n}$. Since $m_s(H) \ge \frac{e_F-1}{v_F-s}$ (recall the definition of $m_s$ given in~\eqref{eq:msH}), we have
  \[
    \deg_{\cH}B \le n^{v_H-v_F} = n^{-(v_F-s)} \cdot n^{v_H-s} \le n^{-\frac{e_F-1}{m_s(H)}} \cdot  n^{v_H-s} = \left(n^{\frac{-1}{m_s(H)}}\right)^{|B|-1} \cdot n^{v_H-s}.
  \]
  Since $B$ was arbitrary, we may conclude that
  \begin{equation}
    \label{eq:Delta-s-H-Kn}
    \Delta_u(\cH) \le \left(n^{-\frac{1}{m_s(H)}}\right)^{u-1} \cdot n^{v_H-s}
  \end{equation}
  for every $u \ge 2$.

  Let $\lambda$ be the constant given by Theorem~\ref{thm:main-hypergraphs} invoked with $K = e_H$ and $\eps_{\textrm{Thm~\ref{thm:main-hypergraphs}}}=\eps/2$ and let $C$ be the larger of the constants given by Theorems~\ref{thm:main-hypergraphs} and~\ref{thm:lower-bound}, also with $\eps_{\textrm{Thm~\ref{thm:lower-bound}}}=\eps/2$. Lastly, let $L = L(\eps, \lambda, C, H)$ be a sufficiently large constant and suppose that $L n^{-1/m_s(H)} \le p \le p_0$.

  By choosing $L$ large, we guarantee that $n$ is large as well and, consequently,
  \[
    \frac{e(\cH)}{v(\cH)} \ge \frac{\binom{n}{v_H}}{\binom{n}{s}} \ge \frac{n^{v_H-s}}{2v_H!}.
  \]
  Together with~\eqref{eq:Delta-s-H-Kn}, this estimate implies that, for every $u \ge 2$,
  \[
    \Delta_u(\cH) \le 2v_H! \cdot \left(\frac{p}{L}\right)^{u-1} \cdot \frac{e(\cH)}{v(\cH)} \le (\lambda p)^{u-1} \cdot \frac{e(\cH)}{v(\cH)},
  \]
  where in the second inequality we used that $L$ is sufficiently large. By Theorems~\ref{thm:main-hypergraphs} and~\ref{thm:lower-bound}, for every $\eta \in [0, 1]$,
  \[
    (1-\eps/2) \cdot \Phi_{n,p}^H(\eta+\eps/2) - C \le  - \log \Pr\big(X \le \eta\Ex[X]\big) \le (1+\eps/2) \cdot \Phi_{n,p}^H\big((1-\eps/2)\eta\big) + C.
  \]

  Finally, we show that we may absorb the additive constant $C$ on both sides of the above inequality. To this end, we first invoke Lemma~\ref{lemma:Phi-p-cH-lower} to get the following inequality:
  \begin{equation}
    \label{eq:Phi-npH-lower}
    \Phi_{n,p}^H (1-\eps/2) \ge \frac{\eps^2}{8e_H^2} \cdot \binom{n}{s} p \ge \frac{L\eps^2}{16e_H^2s!} \ge \frac{2C}{\eps},
  \end{equation}
  where we used the assumptions that $p \ge Ln^{-1/m_s(H)} \ge Ln^{-s}$ and that $L$ is sufficiently large.  To derive the claimed the upper bound on $-\log\Pr(X \le \eta\Ex[X])$, note that, since $\eta \le 1$ and the function $\eta \mapsto \Phi_{n,p}(\eta)$ is decreasing, we have
  \[
    C \stackrel{\eqref{eq:Phi-npH-lower}}{\le} (\eps/2) \cdot \Phi_{n,p}^H\big((1-\eps/2)\eta\big)
    \qquad
    \text{and}
    \qquad
    \Phi_{n,p}^H\big((1-\eps/2)\eta\big) \le \Phi_{n,p}^H\big((1-\eps)\eta\big).
  \]
  To derive the claimed lower bound, we may assume that $\eta + \eps \le 1$, since otherwise $\Phi_{n,p}^H(\eta+\eps) = 0$.  Therefore,
  \[
    C \stackrel{\eqref{eq:Phi-npH-lower}}{\le} (\eps/2) \cdot \Phi_{n.p}^H(\eta+\eps/2)
    \quad
    \text{and}
    \qquad
    \Phi_{n.p}^H(\eta+\eps/2) \le \Phi_{n,p}^H(\eta+\eps).
  \]
  This completes the proof of Theorems~\ref{thm:graphs} and~\ref{thm:hypergraphs}.
\end{proof}
 
\begin{proof}[Proof of Theorems~\ref{thm:APs}]
  Let $k$ be a positive integer and let $\cH$ be the $k$-uniform hypergraph with vertex set $V \coloneqq \br{n}$ whose hyperedges are the $k$-term arithmetic progressions in $\br{n}$, that is,
  \[
    \cH \coloneqq \big\{\{x, x+d, \dotsc, x+(k-1)d\} : x, d \in \br{n}, x+(k-1)d \le n\big\}.
  \]
  Since every number in $\br{n}$ belongs to at most $kn$ many $k$-term arithmetic progressions and every pair of numbers belongs to at most $\binom{k}{2}$ such progressions, we have
  \[
    \Delta_1(\cH) = kn
    \qquad
    \text{and}
    \qquad
    \Delta_k(\cH) \le \dotsb \le \Delta_2(\cH) \le \binom{k}{2}.
  \]
  Moreover, since $\br{n}$ contains at least $c_kn^2$ many $k$-term progressions, for some constant $c_k > 0$, provided that $n \ge k$, we conclude that $e(\cH)\ge c_kn^2$ and hence
  \[
    \Delta_s(\cH) \le K \cdot \left(n^{-\frac{1}{k-1}}\right)^{s-1} \cdot \frac{e(\cH)}{v(\cH)}\qquad\forall s\in\{1,\dotsc,k\}
  \]
  for some constant $K$ that depends only on $k$.  Therefore, when $L n^{-1/(k-1)} \le p \le p_0$ for a sufficiently large constant $L$, we may apply Theorems~\ref{thm:main-hypergraphs} and~\ref{thm:lower-bound} to derive (with a little help from Lemma~\ref{lemma:Phi-p-cH-lower}) the claimed estimate on $-\log \Pr(X \le \eta\Ex[X])$ for every $\eta \in [0,1]$, as in the previous proof. We leave the details to the reader.
\end{proof}

\bibliographystyle{amsplain}
\bibliography{lower-tails}

\end{document}